\documentclass[final,12pt,authoryear]{elsarticle}

\usepackage[a4paper, total={6.5in, 9.5in}]{geometry}
\usepackage{amssymb}
\usepackage{booktabs}
\usepackage{amsmath}
\usepackage{url}
\usepackage{hyperref}
\usepackage{bbm}
\usepackage[ruled,linesnumbered,vlined]{algorithm2e}
\usepackage{xcolor}  
\usepackage{graphicx}
\usepackage{placeins}
\usepackage{amssymb}
\usepackage{amsthm}
\newtheorem{proposition}{Proposition}
\newtheorem{theorem}{Theorem}
\newtheorem{lemma}{Lemma}
\newtheorem{remark}{Remark}
\usepackage{subfig}
\usepackage{subcaption}

\journal{Transportation Research Part B: Methodological}

\begin{document}

\begin{frontmatter}



\title{Modeling Cascading Driver Interventions in Partially Automated Traffic: A Semi-Markov Chain Approach} 

\author[a]{Zihao Li}\
\author[a]{Fan Pu}
\author[b]{Soyoung Ahn\corref{cor1}}
\author[a]{Yang Zhou\corref{cor1}}

\affiliation[a]{organization={Zachry Department of Civil and Environmental Engineering, Texas A\&M University},
            addressline={Engineering Building, 201 Dwight Look}, 
            city={College Station},
            postcode={77840}, 
            state={TX},
            country={USA.}}
\affiliation[b]{organization={Department of Civil and Environmental Engineering, University of Wisconsin Madison},
            addressline={1415 Engineering Drive}, 
            city={Madison},
            postcode={53706}, 
            state={WI},
            country={USA.}}            

\cortext[cor1]{Corresponding author: sue.ahn@wisc.edu, yangzhou295@tamu.edu}


\begin{abstract}
This paper presents an analytical modeling framework for partially automated traffic, incorporating cascading driver intervention behaviors. In this framework, drivers of partially automated vehicles have the flexibility to switch driving modes (either AV or HDV) under lockout constraints. The cascading impact is captured by making the switching probability leader-dependent, highlighting the influence of the leading vehicle on mode choice and the potential propagation of mode changes throughout traffic. Due to the complexity of this system, traditional Markov-based methods are insufficient. To address this, the paper introduces an innovative semi-Markov chain framework with lockout constraints, ideally suited for modeling the system dynamics. This framework reformulates the system as a nonlinear model whose solution can be efficiently approximated using numerical methods from control theory, such as the Runge-Kutta algorithm. Moreover, the system is proven to be a piecewise affine bilinear system, with the existence of solutions and both local and global stability established via Brouwer's Fixed Point Theorem and the 1D Uncertainty Polytopes Theorem. Numerical experiments corroborate these theoretical findings, confirming the presence of cascading impacts and elucidating the influence of modeling parameters on traffic throughput, thereby deepening our understanding of the system's properties.
\end{abstract}



\end{frontmatter}



\section{Introduction}
\vspace{-6pt}

Automated vehicle (AV) technologies have the potential to significantly improve overall traffic flow and efficiency \citep{chen2023stochastic,milanes2013cooperative,shladover2012impacts,talebpour2016influence,chen2017towards,ward2011criteria,swaroop1996string}. Despite notable advances, full automation is many years away \citep{mahmassani201650th, zhou2020stabilizing}. Human drivers remain an integral part of the control loop, giving rise to human factors issues such as trust and willingness to rely on automation. Partial automation (e.g., vehicles equipped with Adaptive Cruise Control, ACC), in which the automated system and the human driver share control of the vehicle, introduces the critical concept of \textit{control transition}. This refers to the transfer of some or all driving tasks between the automated system and the human driver, often triggered by changes in traffic conditions or the driver’s mental or physical state \citep{eriksson2017takeover, mcdonald2019toward}. Control transitions can be classified as either \textit{voluntary} or \textit{non-voluntary}, in terms of who initiates the transition \citep{zhong2025human,gershon2021driver}.  A voluntary takeover occurs when the driver decides to intervene and assume control, influenced by their judgment of the situation or level of trust in the system \citep{eriksson2017takeover,ma2021drivers,bellem2018comfort}. In contrast, a non-voluntary takeover is initiated by the automated system when conditions exceed its operational design domain, requiring the driver to promptly resume control \citep{zhang2019determinants}.

Control transitions can also be categorized by the direction of the control shift: \textit{upward} from manual control to automation, and \textit{downward} for vice versa \cite{correa2018management}. By the nature of the interventions, control transitions accompany sudden changes in car-following behavior (e.g., difference in equilibrium spacing). These sudden changes are highly likely to trigger traffic disturbances \citep{chen2012behavioral, zhong2025human}. This is problematic because once formed, traffic disturbances often grow to full-fledged stop-and-go traffic, undermining traffic throughput and stability \citep{del2001propagation,mauch2002freeway, ahn2007freeway,knoop2008capacity, sugiyama2008traffic,laval2010mechanism, li2011characterization,chen2012behavioral,chen2014periodicity,laval2014parsimonious,oh2015impact, li2024disturbances}. Moreover, when disturbances get amplified along a string of vehicles, they may cause following partially automated vehicles to accordingly switch from automated to manual control (i.e., downward transition). This transition can trigger similar responses in following vehicles, creating a chain reaction or "cascading effect" that spreads through the traffic flow and further reduces overall throughput and stability.

Mixed traffic throughput has been extensively studied in the literature \citep{chen2017towards, ghiasi2017mixed, mohajerpoor2019mixed, qin2025markov, yue2023markov}. \citet{chen2017towards} proposed a theoretical formulation of mixed traffic throughput consisting of automated vehicles (AVs) and human-driven vehicles (HDVs) under equilibrium conditions. \citet{ghiasi2017mixed} developed an analytical stochastic model based on Markov chains to estimate the mixed traffic capacity of HDVs and connected automated vehicles (CAVs), taking into account CAV penetration rate, platoon intensity, and headway settings. Building on this, \citet{mohajerpoor2019mixed} considered the arrangement of AVs and HDVs in mixed traffic to analytically derive the lowest and highest achievable headways and their variability. However, these studies primarily focus on steady-state throughput analysis and overlook critical transient behaviors. Moreover, they typically model scenarios involving fully AVs or idealized CAV platoons, which are not yet common on highways. In practice, partial automation is far more prevalent. A key factor missing in the literature is the treatment of control mode transitions between automated and human-driven modes, along with the cascading effects caused by the chain reactions these transitions may trigger across vehicles. This oversight may lead to overly optimistic assessments of mixed traffic performance.

To address these limitations, this study formulates mixed traffic throughput by capturing control transitions in partial automation systems and their cascading effects. The mixed traffic considered in this study includes both permanent HDVs (i.e., vehicles that are always manually driven) and partially automated vehicles (i.e., vehicles that can switch between automated and manual driving modes). Specifically, we propose a semi-Markov chain framework to model control transitions. Unlike standard Markov models, the semi-Markov approach accounts for deterministic ``lockout” periods, during which vehicles are restricted from switching modes, by allowing transition probabilities to depend not only on the current state but also on the time elapsed in that state. This capability makes the model particularly suited for capturing the time-dependent behavior of control transitions. In addition, we introduce leader-dependent transitions, where the transition probabilities of a following vehicle are conditioned on the mode of its leading vehicle (AV or HDV). This mechanism essentially captures the cascading effects—how control transitions can trigger disturbances that are amplified and affect traffic throughput over time. Building on the proposed leader-dependent semi-Markov system, we conduct an analysis of its fundamental properties, including the existence of equilibrium and system stability. Finally, we conduct numerical experiments to demonstrate the cascading effect of mode transitions as reflected in traffic throughput, and perform a sensitivity analysis on several key factors, including the proportion of permanent HDVs in the traffic flow and the initial mode ratio between HDVs and AVs among partially automated vehicles.

\section{Problem Statement and Assumptions}
\vspace{-6pt}
\label{assumption}
In this framework, we model mixed traffic comprising conventional human-driven vehicles (HDVs) and partially automated vehicles (PAVs), reflecting today’s deployment. Currently, most on-road vehicles with automation functions at SAE Level 2, with a few systems nearing Level 3; fully autonomous Level 4 and 5 vehicles remain exceedingly rare \citep{on2021taxonomy}.  A fraction \(\gamma\) of the vehicles are permanent HDVs, vehicles without any automation functionality, while the remaining fraction \(1-\gamma\) comprises PAVs that can switch between AV and HDV modes. The transition dynamics of PAVs are influenced by the driving behavior and mode of the preceding vehicle (leader-dependent transition).  This dependency can induce cascading effects, whereby disturbances or mode transitions originating from one vehicle propagate downstream, potentially amplifying traffic fluctuations throughout the system, as show in Figure \ref{fig:illustration}. 

Mode transitions of PAV require a deliberate verification process rather than an immediate switch \citep{on2021taxonomy,iso2019pas}. For instance, when transitioning from AV to HDV mode, the vehicle remains in its automated state for several seconds while the system verifies that the driver is attentive and ready to assume control. This transition delay ensures that all safety criteria are met before handing over control to the human driver. Empirical findings \citep{eriksson2017takeover, zhang2019determinants} indicate that drivers typically need three to five seconds to regain situational awareness. Likewise, when moving from HDV to AV mode, the system should assess whether the surrounding environment meets the operational design domain criteria, including verifying acceptable traffic density and sensor clarity to ensure that automation can be safely engaged.  As a result of these transition delays (aka lockout periods), the transition process between control states cannot be accurately represented by a conventional Markov chain \citep{ghiasi2017mixed, qin2025markov, yue2023markov}. Markov models assume memoryless transitions, where the next state depends only on the current state, not on the time spent in it. However, in partial automation, mode switching involves a transition state with a lockout period to satisfy safety and verification requirements. Since the duration spent in this transition state directly determines when a switch can occur, the process is inherently time-dependent and violates the Markov assumption. To address this limitation, we adopt a semi-Markov modeling framework, which explicitly incorporates dwell time in the transition state and better reflects the temporal structure of mode switching in real-world driving. The modeling details are provided in Section~\ref{section:method}.

\begin{figure}[!ht]\centering
    \includegraphics[width=0.9\textwidth]{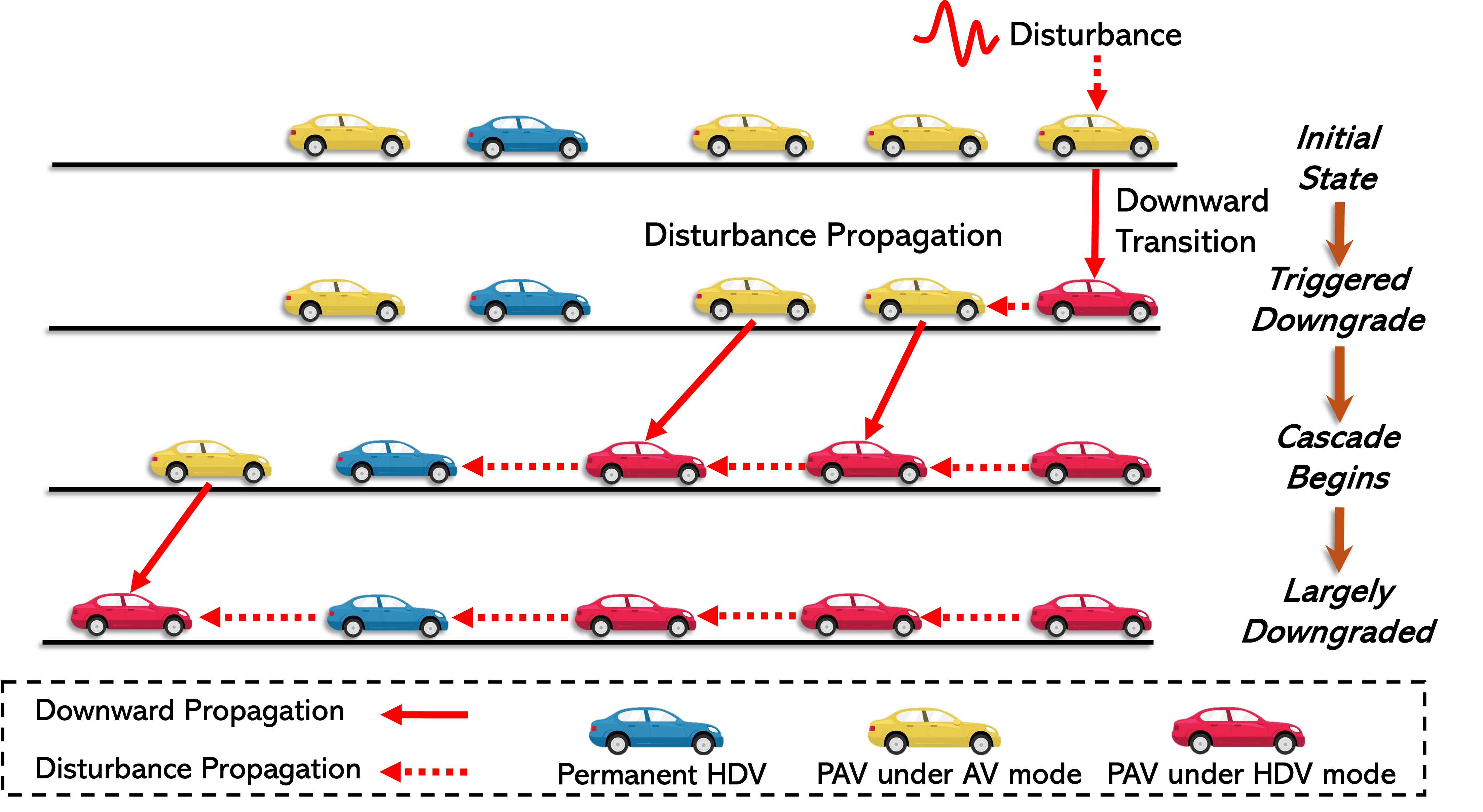}
    \vspace{-5pt} 
    \caption{Downward Mode Transitions (AV mode to HDV mode) and cascading effects in mixed traffic}
    \label{fig:illustration}
\end{figure}
\vspace{-6pt}
The objective of this study is to model the fractions of AVs and HDVs among PAVs over time and compute the resulting traffic throughput, which accounts for both PAVs and permanently assigned HDVs. The analysis incorporates distinct headway requirements for AVs, HDVs, and transitional headway. To formalize the mathematical framework and clearly define the model’s scope, we make the following key assumptions:

\begin{itemize}    
\item \textbf{Leader-Dependent Switching:} The transition of a PAV is modeled as a continuous-time process. Instead of using fixed probabilities as in discrete-time approaches, we characterize transitions by rates \citep{anderson2012continuous}. In the real world, a PAV may switch modes in response to changes in traffic density or the behavior of the leader vehicle. In this study, we focus on the aggregate transition behavior rather than on detailed local traffic conditions that may affect individual transitions. Specifically, rather than modeling how each PAV makes decisions based on detailed local conditions (e.g., speed, headway, or acceleration of the leader), we adopt an approach that captures aggregate transition dynamics using state-dependent rates based on the leader’s status (HDV or AV). If the leader is an HDV, the PAV transitions from HDV to AV mode at rate \(\lambda_1\) and from AV to HDV mode at rate \(\lambda_2\). Conversely, if the leader is an AV, the corresponding transition rates are \(\lambda_3\) (HDV to AV) and \(\lambda_4\) (AV to HDV). 
\item \textbf{Lockout Period:} Before a mode transition occurs, a PAV is required to remain in its current mode for a fixed duration, denoted by \(T_{\text{lock}}\) (or more specifically, \(T^{H}_{\text{lock}}\) for transitions from HDV to AV and \(T^{A}_{\text{lock}}\) for transitions from AV to HDV). Physically, this period allows the vehicle’s control system to verify that safety criteria are met, such as confirming the driver’s alertness or assessing the surrounding environment, before switching modes. This safety verification period also prevents rapid, successive transitions that could lead to driver mode confusion or unsafe operation \citep{colley2022effects}. 
    \item \textbf{Large Population Assumption:} For a sufficiently large number of vehicles, the probability that a given vehicle's leader is an HDV or an AV can be approximated by the overall proportions of HDVs and AVs in the traffic stream.
\end{itemize}

\section{Leader-Dependent Semi-Markov Mixed Traffic System}
\label{section:method}
\vspace{-6pt}
In this section, we develop a semi-Markov modeling framework for a leader-dependent mixed traffic system that incorporates deterministic lockout periods. This framework captures the non-exponential nature of mode transitions, such as the lockout time required to verify safety conditions before switching between manual and automated control. However, the lack of memorylessness in semi-Markov processes poses analytical challenges. To address this, we approximate the deterministic lockout period using a phase-type distribution, enabling the semi-Markov formulation to be converted into a Markov-compatible structure. This approach is more amenable to theoretical analysis while still preserving the essential timing characteristics of the lockout period.

\subsection{State Dynamics}
A central challenge in modeling partial automation lies in representing the lockout period, \(T_{\text{lock}}\), during which a vehicle remains in a transition state and is temporarily prevented from switching modes. This fixed delay violates the memoryless property of standard Markov models, which assume exponentially distributed sojourn times. Because the probability of transitioning depends on the time already spent in the state, the system exhibits semi-Markov behavior. To reconcile this time dependency with a Markovian framework, we approximate the lockout period using a phase-type distribution.

\begin{theorem}[Phase-Type Distributions, \cite{david1987least}]\label{thm:phase-type}
Let \(F\) be any cumulative distribution function on \([0,\infty)\) with finite mean. Then, for every \(\varepsilon > 0\), there exists a phase-type distribution \(F_{\mathrm{PH}}\) such that
\begin{equation}\label{eq:phase-type}
\sup_{t\ge 0} \left|F(t) - F_{\mathrm{PH}}(t)\right| \leq \varepsilon.
\end{equation}
\end{theorem}

This theorem guarantees that any sojourn-time distribution, including a deterministic duration, can be approximated arbitrarily well by a Markov-compatible phase-type distribution. 

\begin{proposition}
\label{prop:erlang_design}
Given a deterministic lockout duration \(T_{\text{lock}}\), the Erlang-\(k\) distribution with rate \(\mu = k / T_{\text{lock}}\) approximates the lockout period with decreasing variance as \(k\) increases. There exists a sufficiently large \(k\) such that the resulting distribution approximates the lockout duration within any desired error \(\varepsilon > 0\).
\end{proposition}

\begin{proof}

In this study, we use an Erlang-\(k\) distribution, defined as the sum of \(k\) identical exponential stages with rate \(\mu\), to approximate the fixed lockout duration in a Markov-compatible form \citep{younes2004solving}. Let \(T\) be the random variable denoting the total time required to complete all \(k\) exponential stages, each with rate \(\mu\). That is, \(T\) represents the stochastic duration of the approximated lockout period.  By setting the expected value of \(T\) to match the deterministic lockout time, the rate \(\mu\) can be determined accordingly.
\begin{equation}
\mathbb{E}[T] = \frac{k}{\mu} = T_{\text{lock}} 
\quad \Rightarrow \quad 
\mu = \frac{k}{T_{\text{lock}}}
\label{eq:mu_expression}
\end{equation}
The corresponding variance of \(T\) is
\begin{equation}
\mathrm{Var}(T) = \frac{k}{\mu^2} = \frac{T_{\text{lock}}^2}{k}
\label{eq:erlang_var}
\end{equation}
which decreases as \(k\) increases, ensuring the distribution to become more tightly concentrated around its mean. As \(k \to \infty\), the Erlang distribution converges to a deterministic lockout time.
\end{proof}

\begin{remark}
A key strength of the phase‐type framework is its flexibility. It is dense in the set of all positive‐valued sojourn‐time distributions. This means that any desired sojourn‐time distribution, including those exhibiting stochastic variability, can be approximated arbitrarily well using phase‐type representations. When the sojourn time is stochastic, additional variability or over‐dispersion may need to be captured beyond what an Erlang distribution offers. In such cases, the hyper-exponential distributions can be considered \citep{bladt2005review}.
\end{remark}

In our model, we focus on PAVs because of their ability to switch between modes, while a fraction \(\gamma\) of vehicles are permanent HDVs that do not switch modes. Following Proposition~\ref{prop:erlang_design}, we introduce \(k\) intermediate states representing sequential exponential stages, allowing the total sojourn time to closely approximate the deterministic lockout time. As shown in Figure~\ref{fig:k-erlang-explain}, the resulting Markov process includes additional states, preserving the Markovian structure while capturing the lockout delay.

\begin{figure}[!ht]\centering
    \includegraphics[width=0.68\textwidth]{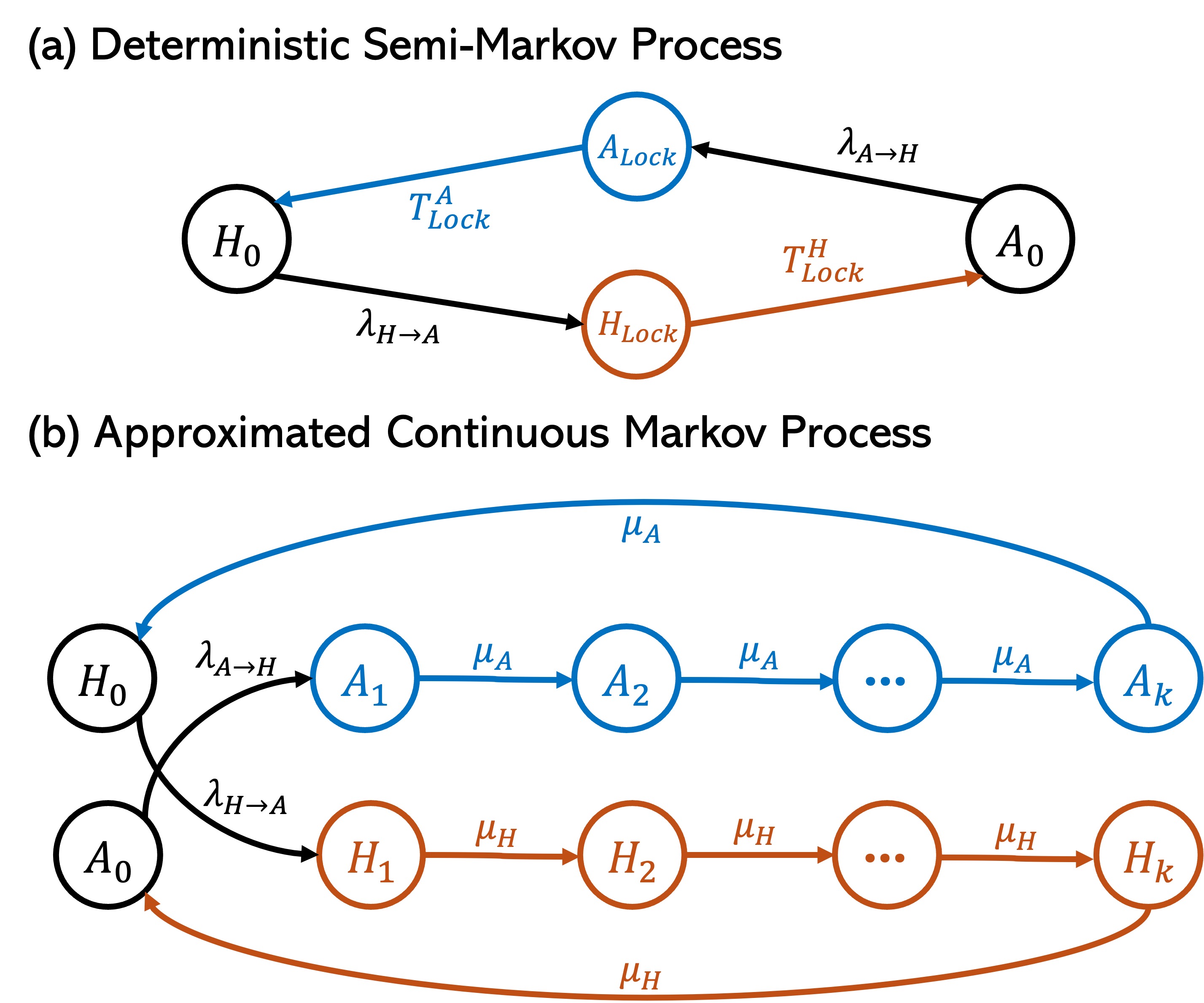}
    \vspace{-5pt} 
    \caption{PAV state transition diagram}
    \label{fig:k-erlang-explain}
\end{figure}
\vspace{-6pt}

We define the system state as 
\begin{equation}\label{eq:state_vector}
x(t) = 
\begin{bmatrix} 
x_{H_0}(t) & x_{H_1}(t) & \cdots & x_{H_k}(t) & x_{A_0}(t) & x_{A_1}(t) & \cdots & x_{A_k}(t)
\end{bmatrix}^T
\end{equation}

where, \(x_{H_i}(t)\) and \(x_{A_i}(t)\) represent the fractions of PAVs in HDV mode and AV mode, respectively, at state \(i\) and time \(t\), where \(i = 0, 1, \dots, k\). When \(i = 0\), the PAV can switch modes. When \(i \neq 0\), it is in the lockout period and need to go through \(k\) steps. When it reaches \(i = k\), the lockout is over and the mode transition is finished. These fractions satisfy the normalization condition, as they collectively represent all PAVs in the system. At any given time \(t\), each PAV occupies exactly one state, ensuring that the total sum of all fractions equals one.

\begin{equation}\label{eq:normalization}
    \sum_{i=0}^{k} x_{H_i}(t) + \sum_{i=0}^{k} x_{A_i}(t) = 1,
\end{equation}

The system dynamics are governed by a generator matrix \(\mathbf{A}\), such that
\begin{equation}\label{eq:state_dynamics}
    \frac{d}{dt}x(t) = \mathbf{A}(t)x(t).
\end{equation}
Because the state vector is naturally partitioned into two groups (HDV and AV mode), the matrix \(\mathbf{A}(t)\) can be expressed in block-matrix form as follows:
\begin{equation}\label{mat_A}
\mathbf{A}(t) = \begin{bmatrix} A_{HH}(t) & A_{HA}(t) \\[1mm] A_{AH}(t) & A_{AA}(t) \end{bmatrix},
\end{equation}
Take upward transition (HDV to AV mode) in Figure~\ref{fig:k-erlang-explain} as example, before transiting to AV mode, the PAV under HDV mode should pass through an intermediate locked HDV state to verify that the surrounding driving environment is safe to activate AV mode.  In our model, the unlocked HDV state \(H_0\) exhibits an outflow at rate \(\lambda_{H\to A}\) to the first locked HDV state \(H_1\), which means that the PAV initiates the transition process toward AV mode. \(H_0\) also receives an inflow from the terminal locked AV state \(A_k\) at a rate of \(\mu_A\) (i.e., this inflow originates from the downward transition, AV to HDV mode).  Each intermediate locked HDV state (\(H_1\) to \(H_k\)) has equal inflow and outflow at rate \(\mu_H = k / T^H_{\text{lock}}\), ensuring balanced progression through the lockout stages. The same structure applies to the AV states in downward transitions.

\begin{subequations}
\begin{equation}\label{mat_1}
A_{HH}=
\begin{pmatrix}
-\lambda_{H\to A} & 0 & 0 & \cdots & 0\\[1mm]
\lambda_{H\to A}  & -\mu_H    & 0 & \cdots & 0\\[1mm]
0                 & \mu_H     & -\mu_H & \cdots & 0\\[1mm]
\vdots            & \vdots    & \ddots & \ddots & \vdots\\[1mm]
0                 & 0         & \cdots & \mu_H  & -\mu_H
\end{pmatrix}
\qquad
A_{HA}=
\begin{pmatrix}
0 & 0 & 0 & \cdots & \mu_A\\[1mm]
0 & 0 & 0 & \cdots & 0\\[1mm]
\vdots & \vdots & \ddots & \ddots & \vdots\\[1mm]
0 & 0 & 0 & \cdots & 0
\end{pmatrix}
\end{equation}

\begin{equation}\label{mat_2}
A_{AH}=
\begin{pmatrix}
0 & 0 & 0 & \cdots & \mu_H\\[1mm]
0 & 0 & 0 & \cdots & 0\\[1mm]
\vdots & \vdots & \ddots & \ddots & \vdots\\[1mm]
0 & 0 & 0 & \cdots & 0
\end{pmatrix}
\qquad
A_{AA}=
\begin{pmatrix}
-\lambda_{A\to H} & 0 & 0 & \cdots & 0\\[1mm]
\lambda_{A\to H}  & -\mu_A    & 0 & \cdots & 0\\[1mm]
0                 & \mu_A     & -\mu_A & \cdots & 0\\[1mm]
\vdots            & \vdots    & \ddots & \ddots & \vdots\\[1mm]
0                 & 0         & \cdots & \mu_A  & -\mu_A
\end{pmatrix}
\end{equation}
\end{subequations}

Thus far, the generator matrix \(\mathbf{A(t)}\) has been expressed in terms of the effective transition rates \(\lambda_{H\to A}(t)\) and \(\lambda_{A\to H}(t)\).  In our formulation, these rates are defined conditionally based on the type of leader a vehicle has. Specifically, the effective rate is calculated as a weighted average of two base rates, where the base rates represent the transition rates under a specific leader type. For example, if a vehicle's leader is an HDV, the transition from HDV to AV occurs at rate \(\lambda_1\). If the leader is an AV, the transition from HDV to AV occurs at rate \(\lambda_3\). The weights are given by the probabilities that the leader is an HDV or an AV, respectively. 
\begin{subequations}\label{h-a&a-h}
\begin{equation}
\lambda_{H\to A}(t) = q_{\mathrm{HDV}}(t)\,\lambda_1 + q_{\mathrm{AV}}(t)\,\lambda_3,
\end{equation}
\begin{equation}
\lambda_{A\to H}(t) = q_{\mathrm{HDV}}(t)\,\lambda_2 + q_{\mathrm{AV}}(t)\,\lambda_4.
\end{equation}
\end{subequations}

In the overall traffic flow, a fraction \(\gamma\) of vehicles are permanent HDVs, while the remaining fraction \(1 - \gamma\) are PAVs whose operating modes may vary over time. Since the PAVs are distributed among the states corresponding to HDV mode (i.e., \(H_0, H_1, \dots, H_k\)) and AV mode (i.e., \(A_0, A_1, \dots, A_k\)), the probability that a vehicle's leader is an HDV at time \(t\) is given by
\begin{equation}
q_{\mathrm{HDV}}(t) = \gamma + (1-\gamma)\sum_{i=0}^{k} x_{H_i}(t),
\end{equation}
and similarly, the probability that the leader is an AV is
\begin{equation}
q_{\mathrm{AV}}(t) = (1-\gamma)\sum_{i=0}^{k} x_{A_i}(t).
\end{equation}

Since the transition rates depend on these probabilities, which are functions of the system state, the generator matrix \(\mathbf{A}\) becomes a function of the state vector, denoted as \(\mathbf{A}(x(t))\). Therefore, the system we propose is a nonlinear system, described by:
\begin{equation}\label{eq:nonlinear_state_dynamics}
    \frac{d}{dt}x(t) = \mathbf{A}(x(t))\,x(t).
\end{equation}

By further rearranging the generator matrix, we can represent the system as a piecewise affine bilinear system. In our setting, this refers to dynamics that are linear in the state vector but weighted by state-dependent probabilities \(q_{\mathrm{HDV}}(x(t))\) and \(q_{\mathrm{AV}}(x(t))\), making the overall system nonlinear. This rearrangement decomposes the nonlinear state dependence into an affine combination of constant matrices, clarifying the influence of each operating mode on the overall dynamics.
\begin{equation}
\frac{d}{dt}x(t) = \Bigl[\,\mathbf{A}_0 + q_{\mathrm{HDV}}(x(t))\,\mathbf{A}_1 + q_{\mathrm{AV}}(x(t))\,\mathbf{A}_2 \Bigr]\,x(t),
\end{equation}
where the constant matrix \(\mathbf{A}_0\) collects the state-independent dynamics. Define as
\begin{equation}\label{eq:A0}
\mathbf{A}_{0} = \begin{bmatrix}
A_{HH}^{(0)} & A_{HA}\\[2mm]
A_{AH}   & A_{AA}^{(0)}
\end{bmatrix},
\end{equation}
with

\begin{equation}\label{eq:AHH0_AAA0}
\begin{aligned}
A_{HH}^{(0)} &= 
\begin{pmatrix}
0      & 0       & 0       & \cdots & 0      \\
0      & -\mu_H & 0       & \cdots & 0      \\
0      & \mu_H  & -\mu_H & \ddots & \vdots \\
\vdots & \vdots & \ddots  & \ddots & 0      \\
0      & 0      & \cdots  & \mu_H  & -\mu_H 
\end{pmatrix}, \qquad
A_{AA}^{(0)} =
\begin{pmatrix}
0      & 0       & 0       & \cdots & 0      \\
0      & -\mu_A & 0       & \cdots & 0      \\
0      & \mu_A  & -\mu_A & \ddots & \vdots \\
\vdots & \vdots & \ddots  & \ddots & 0      \\
0      & 0      & \cdots  & \mu_A  & -\mu_A 
\end{pmatrix}.
\end{aligned}
\end{equation}

Generator matrix \(\mathbf{A}_1\) captures the contribution to the transition dynamics when the leader is an HDV, as in Eq. \ref{h-a&a-h} (i.e., the base rate \(\lambda_1\) for HDV \(\to\) AV and \(\lambda_2\) for AV \(\to\) HDV). Specifically, we set
\begin{equation}\label{eq:A1}
\mathbf{A}_1 = \begin{bmatrix}
A_{HH}^{(1)} & 0\\[2mm]
0 & A_{AA}^{(1)}
\end{bmatrix},
\end{equation}
with
\begin{equation}\label{eq:AHH1_AAA1}
\begin{aligned}
A_{HH}^{(1)} = \begin{pmatrix}
-\lambda_1 & 0 & \cdots & 0\\[1mm]
\lambda_1  & 0 & \cdots & 0\\[1mm]
0         & 0 & \cdots & 0\\[1mm]
\vdots    & \vdots & \ddots & \vdots\\[1mm]
0         & 0 & \cdots & 0
\end{pmatrix},\quad
A_{AA}^{(1)} = \begin{pmatrix}
-\lambda_2 & 0 & \cdots & 0\\[1mm]
\lambda_2  & 0 & \cdots & 0\\[1mm]
0         & 0 & \cdots & 0\\[1mm]
\vdots    & \vdots & \ddots & \vdots\\[1mm]
0         & 0 & \cdots & 0
\end{pmatrix}
\end{aligned}
\end{equation}

Generator matrix \(\mathbf{A}_2\) captures the contribution when the leader is an AV (i.e., the base rate \(\lambda_3\) for HDV \(\to\) AV and \(\lambda_4\) for AV \(\to\) HDV). Define
\begin{equation}\label{eq:A2}
\mathbf{A}_2 = \begin{bmatrix}
A_{HH}^{(2)} & 0\\[2mm]
0 & A_{AA}^{(2)}
\end{bmatrix},
\end{equation}
with
\begin{equation}\label{eq:AHH2_AAA2}
\begin{aligned}
A_{HH}^{(2)} = \begin{pmatrix}
-\lambda_3 & 0 & \cdots & 0\\[1mm]
\lambda_3  & 0 & \cdots & 0\\[1mm]
0         & 0 & \cdots & 0\\[1mm]
\vdots    & \vdots & \ddots & \vdots\\[1mm]
0         & 0 & \cdots & 0
\end{pmatrix},\quad
A_{AA}^{(2)} = \begin{pmatrix}
-\lambda_4 & 0 & \cdots & 0\\[1mm]
\lambda_4  & 0 & \cdots & 0\\[1mm]
0         & 0 & \cdots & 0\\[1mm]
\vdots    & \vdots & \ddots & \vdots\\[1mm]
0         & 0 & \cdots & 0
\end{pmatrix}.
\end{aligned}
\end{equation}
Furthermore, by the fact that,
\begin{equation}
q_{\mathrm{HDV}}(t)+q_{\mathrm{AV}}(t)=1
\end{equation}
we can further organize representation of system dynamic:
\begin{equation}\label{finalsystem}
\frac{d}{dt}x(t) = \Bigl[\,\mathbf{A}_0 + q_{\mathrm{HDV}}(x(t))\,\mathbf{A}_1 + \Bigl(1 - q_{\mathrm{HDV}}(x(t))\Bigr)\,\mathbf{A}_2 \Bigr]\,x(t).
\end{equation}

\begin{remark}
By Eq.~\ref{finalsystem}, if the transition rates from AV to HDV and from HDV to AV are independent of the leader, then the factor \(q_{\mathrm{HDV}}(t)\) does not affect the transition dynamics. In this case, the generator matrix \(\mathbf{A}\) is constant, and the system described by Eq.~\ref{eq:nonlinear_state_dynamics} reduces to a linear system, equivalent to a traditional continuous-time Markov chain, whose system properties analysis such as steady-state behavior and overall stability is much easier.
\end{remark}

\begin{remark}\label{AV-HDV}
A transition from AV mode to HDV mode may occur either voluntarily or involuntarily. In a voluntary takeover, the vehicle initiates a lockout period during which it remains in AV mode while verifying that the driver is ready to assume control—this behavior is captured in our model. In contrast, an involuntary takeover occurs under emergency conditions, and proceeds immediately without a lockout delay, as shown in Figure~\ref{fig:No_lock}. This special case is still represented by the proposed framework: as \(T_{\text{lock}}^A \to 0\), the corresponding transition rate satisfies \(\mu_A \to \infty\), according to Proposition~\ref{prop:erlang_design}.

\begin{figure}[!ht]\centering
    \includegraphics[width=0.65\textwidth]{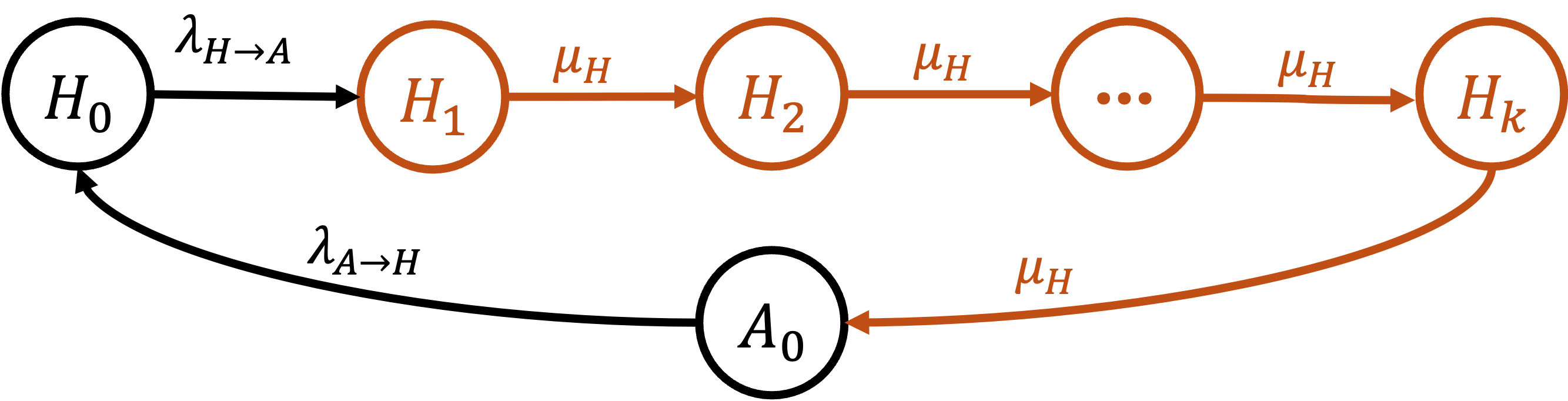}
    \vspace{-5pt} 
    \caption{PAV state transition diagram without downward transition lockout period}
    \label{fig:No_lock}
\end{figure}
\vspace{-6pt}
\end{remark}

\section{System Property Analysis}
\vspace{-6pt}
Building on the model presented above, this section analyzes system properties, focusing on the existence of steady-state solutions and the stability of the proposed dynamics. This analysis provides a comprehensive understanding of the mixed traffic system, particularly under leader-dependent mode transitions. Finally, by incorporating the fraction of vehicles in each state along with their respective headway values, we derive the time-varying traffic throughput. This derivation highlights how mode transitions among PAVs lead to disturbance propagation and throughput degradation.

\subsection{Steady State Analysis}
In traditional continuous Markov systems, the state vector is represented as a column vector and the generator matrix is Metzler, ensuring conservation and typically leading to a unique, globally attractive steady state \citep{norris1998markov}. In our system, the dynamics are more complex because the generator matrix depends on the state. It is formed by combining a base transition matrix with additional contributions that depend on whether the leading vehicle is in HDV or AV mode, as shown in Eq.\ref{finalsystem}. Since the generator matrix varies with the state, the system becomes both time-varying and nonlinear, making traditional eigenvalue-based stability analysis impractical \citep{sontag2013mathematical}. To address these challenges, we first apply Brouwer’s Fixed-Point Theorem to show that an equilibrium exists. Next, we exploit the conservation property to reduce the system and remove redundant system state. Finally, we use robust stability theory based on a common quadratic Lyapunov function to establish global exponential stability.

\subsubsection{Existence of Equilibrium}
\vspace{2pt}

\begin{theorem}[Brouwer’s Fixed-Point Theorem, \cite{brouwer1911abbildung}]\label{Brouwer}
Let \(T: X \to X\) be a continuous mapping on a nonempty compact convex set $X$. Then there exists at least one point $x^*\in X$ such that $T(x^*)=x^*$, that is, $T$ admits fixed points in $X$.
\end{theorem}



Our goal is to demonstrate that the fixed point guaranteed by Brouwer’s Fixed-Point Theorem corresponds to the equilibrium $x^*$ of our system, which also satisfies $\mathbf{A}(x^*)x^*=0$. First, we show that the feasible state space $\mathbf{X}$ is a non-empty compact convex set (see Remark \ref{x-convex}). Next, we define a mapping $T$ and show that it is a self-map that maps $\mathbf{X}$ into itself (Propositions \ref{F_positive} - \ref{T_selfmap}). Finally, we show that $T$ warrants the existence of a state $x^*\in\mathbf{X}$ such that $T(x^*) = x^*$, which further satisfies $\mathbf{A}(x^*)x^*=0$ (see Lemma \ref{steady_state}).

\begin{remark}\label{x-convex}
The state space is defined as
\begin{equation}
\label{x-range}
\mathbf{X} = \{ x \in \mathbb{R}^n \mid x_i \geq 0,\; \sum_{i=1}^{n} x_i = 1 \},
\end{equation}
which corresponds to the standard \((n-1)\)-simplex in $\mathbb{R}^n$. It is evidently non-empty. To establish compactness, note that each constraint $x_i\ge 0$ defines a closed half-space, and the equality $\sum_{i=1}^n x_i=1$ is a closed hyperplane. As $\mathbf{X}$ is the intersection of finitely many closed sets in $\mathbb{R}^n$, it is itself closed. Moreover, since $x_i\ge 0$ and $\sum_{i=1}^n x_i=1$, each coordinate $x_i\le 1$, implying that $\mathbf{X}$ is bounded. Hence, the closed bounded set $\mathbf{X}$ is compact. To verify convexity, let $x,y\in\mathbf{X}$ and $\lambda\in [0,1]$, then the convex combination of $x$ and $y$ is $z=\lambda x+(1-\lambda)y$. It is easy to check that each coordinate $z_i=\lambda x_i+(1-\lambda)y_i\ge 0$ and
\begin{equation}
    \sum_{i=1}^n z_i = \lambda\sum_{i=1}^nx_i + (1-\lambda)\sum_{i=1}^n y_i = \lambda+(1-\lambda)=1,
\end{equation}
which implies $z\in\mathbf{X}$. Therefore, $\mathbf{X}$ is a non-empty compact convex set.
\end{remark}


Next, we define a mapping $T$ and prove it is a self-map on $\mathbf{X}$.

\begin{proposition}\label{F_positive}
At time \(t\), for any state \(x(t) \in \mathbf{X}\), we define \(F(x(t)) = \dot{x}(t)\) and denote the \(i\)-th elements of \(F(x(t))\) and \(x(t)\) by \(F_i(x(t))\) and \(x_i(t)\), respectively. Then, if \(x_i(t) = 0\), it follows that \(F_i(x(t)) > 0\), and if \(x_i(t) = 1\), then \(F_i(x(t)) < 0\).
\end{proposition}

\begin{proof}
From equations (\ref{mat_A})--(\ref{mat_2}), the off-diagonal entries \(\mathbf{A}_{ij}\) (for \(i \neq j\)) of the matrix \(\mathbf{A}(x(t))\) are positive. Thus, the \(i\)-th component of \(F(x(t))\) can be written as
\begin{equation}
F_i(x(t)) = \dot{x_i}(t)=\sum_{j=1}^n \mathbf{A}_{ij}\, x_j(t)
= \mathbf{A}_{ii}\, x_i(t) + \sum_{\substack{j=1 \\ j \neq i}}^n \mathbf{A}_{ij}\, x_j(t).
\end{equation}

If \(x_i(t) = 0\), then
\begin{equation}
F_i(x(t)) = \sum_{\substack{j=1 \\ j \neq i}}^n \mathbf{A}_{ij}\, x_j(t)
\end{equation}
Since \(\mathbf{A}_{ij} > 0\) for \(j \neq i\) and the components \(x_j(t)\) are nonnegative with \(\sum_{j=1}^n x_j(t)=1\). Hence, \(F_i(x(t)) > 0\).

If \(x_i(t) = 1\), then $x_j(t)=0$ for $j\neq i$ and thus
\begin{equation}
F_i(x(t)) = \mathbf{A}_{ii}
\end{equation}
Since \(\mathbf{A}_{ii} < 0\), it follows that \(F_i(x(t)) < 0\).
\end{proof}

Proposition~\ref{F_positive} guarantees that \(F(x(t))\) is inward pointing along the boundary of the state space. In other words, if a component \(x_i(t)\) is at its lower bound (i.e., \(x_i(t)=0\)), then its derivative \(F_i(x(t))\) is positive, which pushes the state into the interior. Conversely, if \(x_i(t)\) is at its upper bound (i.e., \(x_i(t)=1\)), then \(F_i(x(t))\) is negative, again directing the state into the interior. 




\begin{proposition}\label{F_tangent}
    For any $x(t)\in\mathbf{X}$ at time $t$, $\sum_{i=1}^n F_i(x(t)) = 0$ always holds. 
\end{proposition}

\begin{proof}
    Since the generator matrix $\mathbf{A}(x(t))$ is a Metzler (i.e., its off-diagonal entries are nonnegative), and satisfies $\sum_{i=1}^n\mathbf{A}_{ij}=0$ for every column $j$. We have 
    \begin{equation}
        \begin{aligned}
            \sum_{i=1}^{n} F_i(x(t)) &= \sum_{i=1}^{n} \dot{x}_i(t) = \sum_{i=1}^{n} \left[ \mathbf{A}(x(t))x(t) \right] \\
            &= \sum_{i=1}^{n}\sum_{j=1}^n \mathbf{A}_{ij} x_j(t) = \sum_{j=1}^n \left( \sum_{i=1}^{n} \mathbf{A}_{ij}\right) x_j(t)=0
        \end{aligned}
    \end{equation}

Proposition~\ref{F_tangent} shows that \(F(x(t))\) preserves the total sum of the components of \(x(t)\). That is, if the initial state \(x(0) \in \mathbf{X}\) satisfies \(\sum_{i=1}^n x_i(0)=1\), then for every \(t > 0\) the state \(x(t)\) remains in \(\mathbf{X}\) with \(\sum_{i=1}^n x_i(t)=1\). \end{proof}

Based on Propositions~\ref{F_positive} and~\ref{F_tangent}, we can construct the mapping \(T\) in Proposition \ref{T_selfmap}.

\begin{proposition}\label{T_selfmap}
For a sufficiently small nonnegative constant \(\alpha\), the mapping \(T: \mathbf{X} \to \mathbf{X}\) defined by
\begin{equation}\label{map_def}
T(x(t)) := x(t) + \alpha F(x(t))
\end{equation}
is a continuous self-map on \(\mathbf{X}\); that is, \(T(x(t)) \in \mathbf{X}\) for all \(x(t) \in \mathbf{X}\).
\end{proposition}

\begin{proof}
First, we demonstrate that \(T\) preserves the sum of the components. For any \(x(t) \in \mathbf{X}\), we have
\begin{equation}\label{eq:T_sum_preserved}
\sum_{i=1}^n T_i(x(t)) = \sum_{i=1}^n \Bigl( x_i(t) + \alpha\,F_i(x(t)) \Bigr)
= \sum_{i=1}^n x_i(t) + \alpha \sum_{i=1}^n F_i(x(t)) = 1 + \alpha\cdot 0 = 1.
\end{equation}
Thus, the total sum of the components remains one.

Then, we verify non-negativity. If \(x_i(t) = 0\), Proposition~\ref{F_positive} ensures that \(F_i(x(t)) > 0\), so \(T_i(x(t)) > 0\) for any \(\alpha > 0\). More generally, if \(x_i(t) > 0\) and \(F_i(x(t))<0\), then by continuity and the boundedness of \(F\) on the compact set \(\mathbf{X}\), we can choose \(\alpha\) sufficiently small so that

\begin{equation}\label{eq:T_nonnegative}
T_i(x(t)) = x_i(t) + \alpha\,F_i(x(t)) \ge 0 \quad \text{for all } i.
\end{equation}
Since both the non-negativity condition (Eq.~\ref{eq:T_nonnegative}) and the sum condition (Eq.~\ref{eq:T_sum_preserved}) are satisfied, it follows that \(T(x(t)) \in \mathbf{X}\) (Eq. \ref{x-range}). Therefore, \(T\) is a self-map on \(\mathbf{X}\). 
\end{proof}



Based on the above analysis, we now demonstrate that the fixed point $x^*$ of the mapping $T$, guaranteed by Brouwer's Fixed-Point Theorem, corresponds to the equilibrium of our system.

\begin{lemma}\label{steady_state}
According to Brouwer's Fixed-Point Theorem, since the feasible state space \(\mathbf{X}\) is non-empty, compact, and convex, and the mapping \(T\) defined in Eq.~(\ref{map_def}) is a continuous self-map on \(\mathbf{X}\), there exists at least one fixed point \(x^*\in\mathbf{X}\) satisfying \(T(x^*) = x^*\). Since \(T(x) = x + \alpha F(x)\) for \(\alpha > 0\), the condition \(T(x^*) = x^*\) implies that \(x^* = x^* + \alpha F(x^*)\), which in turn yields \(F(x^*) = 0\). Consequently, \(\mathbf{A}(x^*)x^* = 0\), and thus \(x^*\) is an equilibrium of system (\ref{eq:nonlinear_state_dynamics}).
\end{lemma}

\subsubsection{Stability Analysis}
In our system, because the sum of the system states equal to one, the generator matrix is inherently singular, indicating that one coordinate is redundant. By leveraging this conservation property, we eliminate the redundant state and thereby obtain a reduced system that captures the essential dynamics  \citep{norris1998markov}. we have 
\begin{equation}
\label{eq:dyn_reduced_i}
\dot{x}_i(t) = \sum_{j=0}^{n-1} \mathbf{A}_{ij}\,x_j(t) + \mathbf{A}_{in}\,x_n(t)
\end{equation}
Since the last state satisfies $x_n(t)= 1 -\sum_{j=0}^{n-1} x_j(t)$, we can express dynamics solely in terms of the first \(n-1\) states, as following
\begin{equation}
\dot{x}_i(t) =\; \sum_{j=0}^{n-1} \mathbf{A}_{ij}\,x_j(t) + \mathbf{A}_{in}\,\Bigl[\,1 - \sum_{j=0}^{n-1} x_j(t)\Bigr]
\end{equation}
Identifying \(y_i(t)=x_i(t)\) for \(i=0,\dots,n-1\), we can rearrange the terms to obtain
\begin{equation}
\dot{y}_i(t) = \sum_{j=0}^{n-1} \Bigl(\mathbf{A}_{ij} - \mathbf{A}_{in}\Bigr)y_j(t) + \mathbf{A}_{in}
\end{equation}
By setting $[\mathbf{A}']_{ij} =\mathbf{A}_{ij} - \mathbf{A}_{in}
\;\text{for } i,j=0,\dots,n-1$, and $\mathbf{c}=[\mathbf{A}_{0n},\mathbf{A}_{1n},...,\mathbf{A}_{n-1,n}]^T$, we have 
\begin{equation}
\dot{y}(t) = \mathbf{A}'\,y(t) + \mathbf{c}
\end{equation}
To analyze the stability of this reduced system, we define a perturbation \(z(t)\) as
\begin{equation}\label{eq:z_def}
z(t) = y(t) - y^*
\end{equation}
This definition represents the deviation of the state from its equilibrium. Since \(y^*\) is constant, we have
\begin{equation}\label{eq:z_dot}
\dot{z}(t) = \dot{y}(t) = \mathbf{A}'\,y(t) + \mathbf{c}.
\end{equation}
Substituting \(y(t) = y^* + z(t)\) into \eqref{eq:z_dot} yields
\begin{equation}\label{eq:substitution}
\dot{z}(t) = \mathbf{A}'\bigl(y^* + z(t)\bigr) + \mathbf{c} = \mathbf{A}'\,y^* + \mathbf{A}'\,z(t) + \mathbf{c}.
\end{equation}
Since \(y^*\) satisfies \(\mathbf{A}'\,y^* + \mathbf{c} = 0\), the terms \(\mathbf{A}'\,y^*\) and $\mathbf{c}$ cancel, leaving
\begin{equation}\label{eq:final_z_dot}
\dot{z}(t) = \mathbf{A}'\,z(t).
\end{equation}
As indicated in Eq.~\ref{finalsystem},
\begin{equation}\label{A-pr}
A' = A'_0 + q_{\mathrm{HDV}}(t)\,A'_1 + \bigl(1 - q_{\mathrm{HDV}}(t)\bigr)\,A'_2
\end{equation}
where we use $q_{\mathrm{HDV}}(t)$ instead of $q_{\mathrm{HDV}}(x(t))$ to simplify the notation.

\begin{proposition}\label{global}
Suppose there exists a common positive definite matrix \(P \succ 0\) such that
\begin{equation}\label{Prop:polytope_stability}
M_0^\top P + P\,M_0 \prec 0 \quad \text{and} \quad M_1^\top P + P\,M_1 \prec 0,
\end{equation}
where \(M_0 = A'_0 + A'_2\) and \(M_1 = A'_0 + A'_1\). Then, for every \(q_{\mathrm{HDV}}(t) \in [0,1]\), both the reduced system \(z(t)\) and the proposed system \(x(t)\) are robustly Hurwitz, ensuring global exponentially stability.
\end{proposition}

\begin{proof}
\(A'\) is an affine combination of constant matrices, so we can rewrite it as a one-dimensional uncertainty polytope. By defining \(M_0 = A'_0 + A'_2\) and \(M_1 = A'_0 + A'_1\)
we have
\begin{equation}\label{eq:Aprime_polytope}
A'(q_{\mathrm{HDV}}(t)) = M_0 + q_{\mathrm{HDV}}(t)\Bigl(M_1 - M_0\Bigr).
\end{equation}
As \(q_{\mathrm{HDV}}(t)\) varies in the interval \([0,1]\), the matrix \(A'(q_{\mathrm{HDV}}(t))\) traces a straight line in the space of matrices between the two extreme cases \(M_0\) (when \(q_{\mathrm{HDV}}(t)=0\)) and \(M_1\) (when \(q_{\mathrm{HDV}}(t)=1\)). In other words, every matrix in the family is a convex combination of \(M_0\) and \(M_1\).

Based on a common quadratic Lyapunov function in robust control \citep{zhou1996robust}, if there exists a common positive definite matrix \(P \succ 0\) such that \(M_0^\top P + P\,M_0 \prec 0\) and \(M_1^\top P + P\,M_1 \prec 0\), then for every \(q_{\mathrm{HDV}}(t) \in [0,1]\),
\begin{equation}\label{eq:polytope_stability}
A'(q_{\mathrm{HDV}}(t))^\top P + P\,A'(q_{\mathrm{HDV}}(t)) \prec 0.
\end{equation}

That is, every matrix in the family is Hurwitz, which guarantees global stability of the reduced dynamics. Since the reduced system is equivalent to the proposed system, it follows that \(x(t)\) is also globally stable.\end{proof}

\begin{remark}
The sufficient stability condition expressed by the linear matrix inequality in Eq.~\ref{Prop:polytope_stability} can be efficiently solved using numerical tools such as CVX \citep{cvx,gb08} or YALMIP \citep{Lofberg2004}.
\end{remark}

Therefore, based on Lemma~\ref{steady_state} and Proposition~\ref{global}, the system defined by Eq.~(\ref{finalsystem}) admits a unique equilibrium, which is globally exponentially stable.

\subsection{Traffic Throughput Analysis}
To model the disturbances induced by control mode transitions between AV and HDV, we characterize disturbance propagation through traffic throughput, which is determined by the effective average headway in a mixed traffic stream. The traffic stream consists of two types of vehicles, including PAVs and permanent HDV. PAVs can switch between HDV and AV modes in response to surrounding traffic conditions. To more accurately estimate traffic throughput, it is important to account for changes in headway during mode transitions. We model this headway transition using a general sigmoid function to reflect the gradual shift in vehicle behavior \citep{shao2017sigmoid,chen2024sigmoid}. Our framework adopts an Erlang-\(k\) distribution to approximate the continuous transition time during mode changes. Accordingly, we use a piecewise constant approximation, based on mean values, to discretize the continuous headway transition into \(k\) segments, as shown in Figure \ref{fig:piece}. Then, the corresponding headway associated with each system state of PAV is obtained. In the unlocked states—after a PAV has completed its mode transition—the headway is set to the average equilibrium headway of either the AV mode or the HDV mode, depending on the current mode. For permanent HDVs, the headway is always given by the average HDV headway.

\begin{figure}[!ht]\centering
    \includegraphics[width=0.6\textwidth]{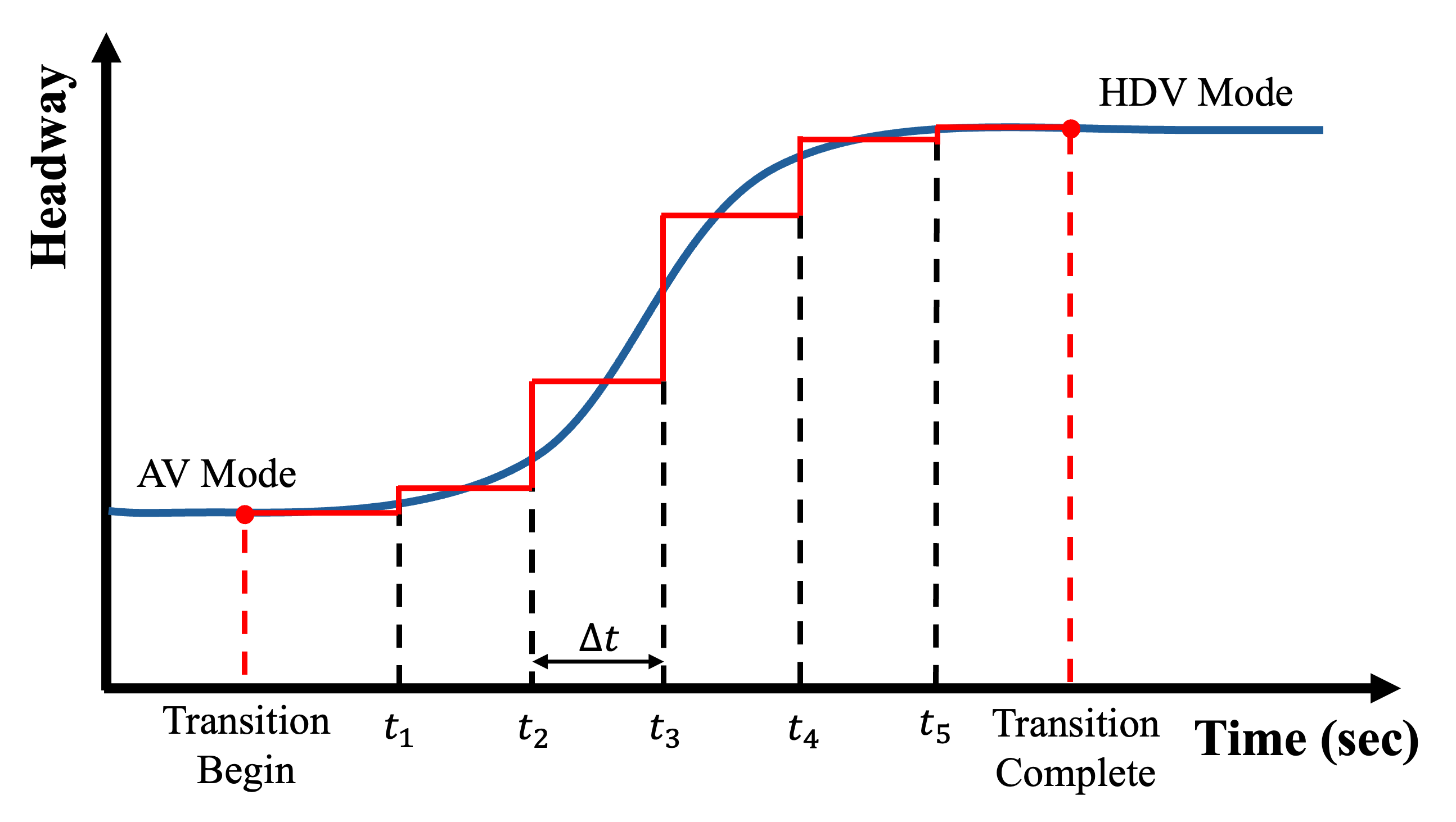}
    \vspace{-12pt} 
    \caption{Example of  piecewise constant approximation for sigmoid headway transition function}
    \label{fig:piece}
\end{figure}
\vspace{-6pt}

The effective average headway is computed as the weighted sum of state-specific headway, based on the fraction of vehicles in each state, as follows:
\begin{equation}
  h_{\mathrm{eff}}(t) = 
  (1 - \gamma)\,
  \underbrace{\sum_{i=0}^{k} \big[ x_{H_i}(t)\,h_{H_i}(t) + x_{A_i}(t)\,h_{A_i}(t) \big]}_{\text{Average Headway of PAV}}
  + \gamma\,
  \underbrace{h_{H_0}(t)}_{\substack{\text{Average} \\ \text{Headway of}\\ \text{Permanent HDV}}}
  \label{eq:tau_eff}
\end{equation}

where \(h_{H_i}(t) = \tau_{H_i}(t) + L_{H_i}(t)/v(t)\) and \(h_{A_i}(t) = \tau_{A_i}(t) + L_{A_i}(t)/v(t)\). The time gaps \(\tau_{H_i}(t)\) and \(\tau_{A_i}(t)\), as well as the standstill distances \(L_{H_i}(t)\) and \(L_{A_i}(t)\), are determined using a piecewise approximation, as shown in Figure \ref{fig:piece}. \(h_{H_0}(t) = \tau_{H_0} + L_{H_0}/v(t)\), where \(\tau_{H_0}\) and \(L_{H_0}\) represent the equilibrium time gap and standstill distance of HDVs, respectively. \(v(t)\) denotes the average traffic speed at time \(t\). Since traffic flow capacity is defined as the number of vehicles passing a point per unit time, it is given by the inverse of the effective headway \citep{chen2023stochastic}. Accordingly, the instantaneous traffic throughput, denoted by \(C(t)\) (vehicles per unit time per lane), is expressed as:

\begin{equation}
  C(t) = \frac{1}{h_{\mathrm{eff}}(t)}
  \label{eq:traffic_flow_capacity}
\end{equation}

By the existence of an equilibrium state, we define the following lemma for the steady-state throughput:

\begin{lemma}
Assume the system reaches an equilibrium as \(t \to \infty\), such that the state vector converges to a fixed point \(x^*\) and the average traffic speed approaches a constant \(v^*\). Then, the steady-state throughput of mixed traffic, consisting of PAVs and permanent HDVs, is given by
\begin{equation}
  C_{\infty} 
       = \frac{1}{(1 - \gamma)\,\sum_{i=0}^{k} \big( x_{H_i}^*\,h_{H_i}^* + x_{A_i}^*\,h_{A_i}^* \big) + \gamma\,h_{H_0}^*},
  \label{eq:steady_traffic_flow_capacity}
\end{equation}
where \(h_{H_i}^* = \tau_{H_i}^* + L_{H_i}^*/v^*\) and \(h_{A_i}^* = \tau_{A_i}^* + L_{A_i}^*/v^*\) are the equilibrium headways in HDV and AV modes, respectively, and \(h_{H_0}^*\) is the equilibrium headway of permanent HDVs.
\end{lemma}

\section{Numerical Experiments}
\vspace{-6pt}
Due to the nonlinearity and state-dependent structure of the generator matrix, we employ numerical methods (i.e., Runge-Kutta fourth-order method \cite{burden1997numerical}) to solve the system and evaluate traffic throughput over time. We begin by selecting a suitable \(k\) to approximate the deterministic lockout period using an Erlang-\(k\) distribution. We then verify the global stability of the system via Proposition~\ref{global}. After that, we examine how leader-dependent transitions trigger mode changes in following PAVs and affect throughput under different traffic scenarios. Finally, we conduct sensitivity analyses on key parameters, such as the permanent HDV rate and the initial fractions of HDVs and AVs among PAVs, to evaluate their effects on throughput. The default parameters used in the numerical experiments are summarized in Table~\ref{tab:parameters}.

\begin{table}[h!]
\centering
\caption{Default parameters for numerical experiments}
\vspace{-5pt}
\begin{tabular}{lll}
\toprule
\textbf{Parameter} & \textbf{Notation}& \textbf{Default Value} \\
\midrule
AV equilibrium time gap& \(\tau_{A_0}\)& 1.0 (s)  \\
HDV equilibrium time gap& \(\tau_{H_0}\)& 1.5 (s)\\
Equilibrium standstill distance for AV& \(L_{A_{0}}\)& 5.0 (m)\\
Equilibrium standstill distance for HDV& \(L_{H_{0}}\)&7.0 (m)\\
Average Speed& \(v\)&$[0,30]$ (m/s)\\
Lockout period for downward transition& \(T^{A}_{\text{lock}}\)&3.0 (s)  \\
Lockout period for upward transition& \(T^{H}_{\text{lock}}\) &3.0 (s)  \\
Simulation duration & \(T\) & 30.0 (s)\\
Discrete time steps & \(h\) & 0.01 (s)  \\
Fraction of permanent HDVs & \(\gamma\)         &$[0,1]$\\
Initial mode fractions of PAVs& $x_{H_0}(0)$/$x_{A_0}(0)$&0.5/0.5\\
HDV→AV rate (leader HDV)  & \(\lambda_1\)      & $(0,1]$ (1/s)\\
AV→HDV rate (leader HDV)  & \(\lambda_2\)      & $(0,1]$ (1/s)\\
HDV→AV rate (leader AV) & \(\lambda_3\)      & $(0,1]$ (1/s)\\
AV→HDV rate (leader AV) & \(\lambda_4\)      & $(0,1]$ (1/s)\\
\bottomrule
\end{tabular}
\label{tab:parameters}
\end{table}
\newpage
\subsection{Erlang-\(k\) Approximation for Deterministic Lockout Time}
To facilitate analysis, we approximate the deterministic lockout period using an Erlang-\(k\) distribution (see Proposition~\ref{prop:erlang_design}). To balance approximation accuracy and computational efficiency, it is important to select an appropriate value for \(k\). We quantify the approximation error using the Wasserstein distance between the deterministic lockout time and its Erlang-\(k\) approximation. While no universal threshold exists, a Wasserstein distance below 0.2 is commonly considered acceptable for distributional similarity \citep{arjovsky2017wasserstein}.

The deterministic lockout time, in terms of its probability density, is modeled as a Dirac delta function. As \(k\) increases, the variance of the Erlang-\(k\) distribution decreases, improving its approximation of the fixed delay (Figure~\ref{fig:k-erlang}(a)). This improvement is also reflected in the decreasing Wasserstein distance, as shown in Figure~\ref{fig:k-erlang}(b), where the red stars indicate the 0.2 threshold. These results suggest that when \(k \geq 200\), the Erlang-\(k\) distribution sufficiently approximates the deterministic lockout period and thus enables the semi-Markov process to be treated within a Markovian framework Accordingly, we set \(k = 200\) as the default value in the following numerical experiment.

\begin{figure}[!h]
    \vspace{-12pt}
    \centering
    \setlength{\abovecaptionskip}{4pt}
    \setlength{\belowcaptionskip}{-6pt}    
    \subfloat[]{%
        \includegraphics[height=0.2\textheight]{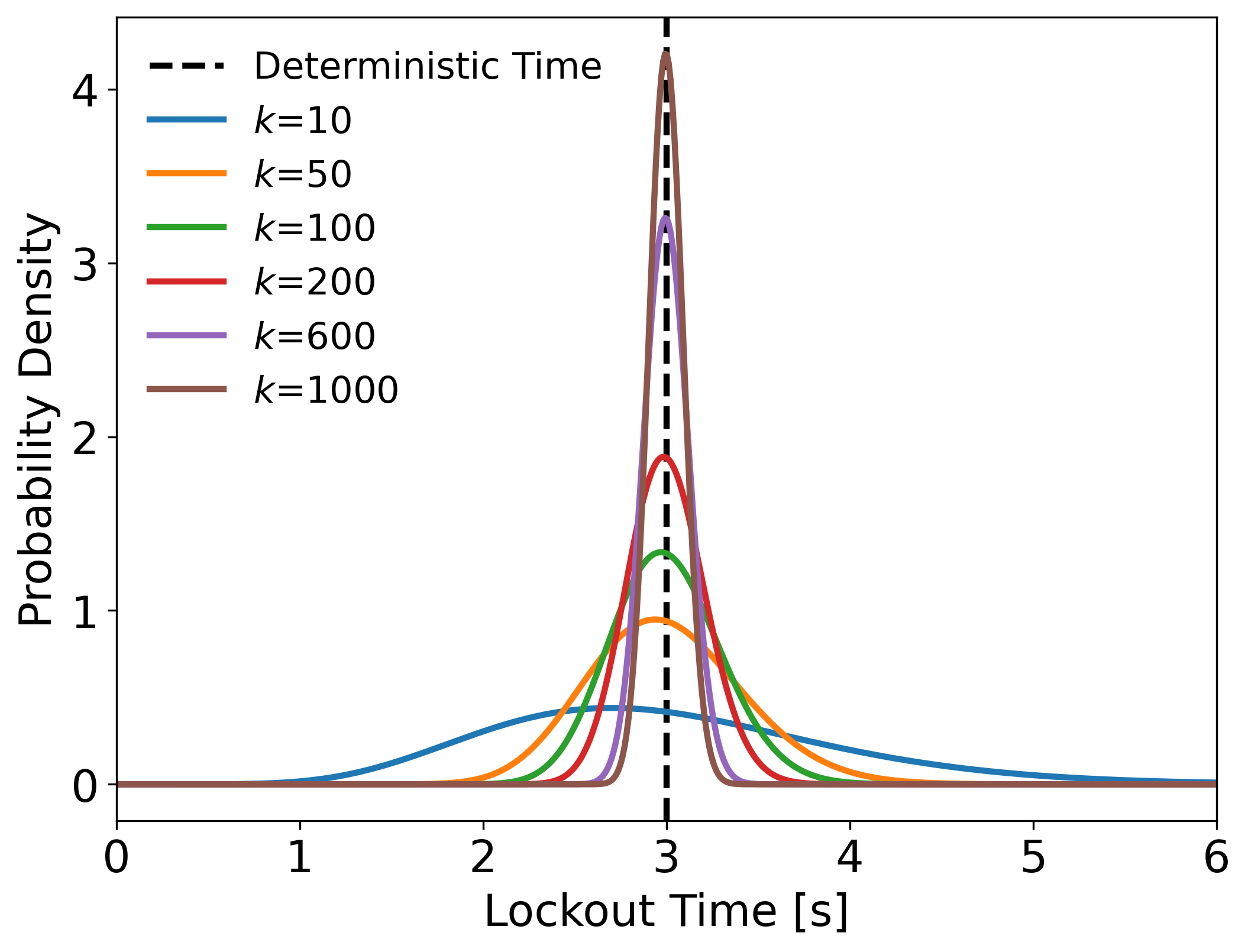}
    }
    \subfloat[]{%
        \includegraphics[height=0.2\textheight]{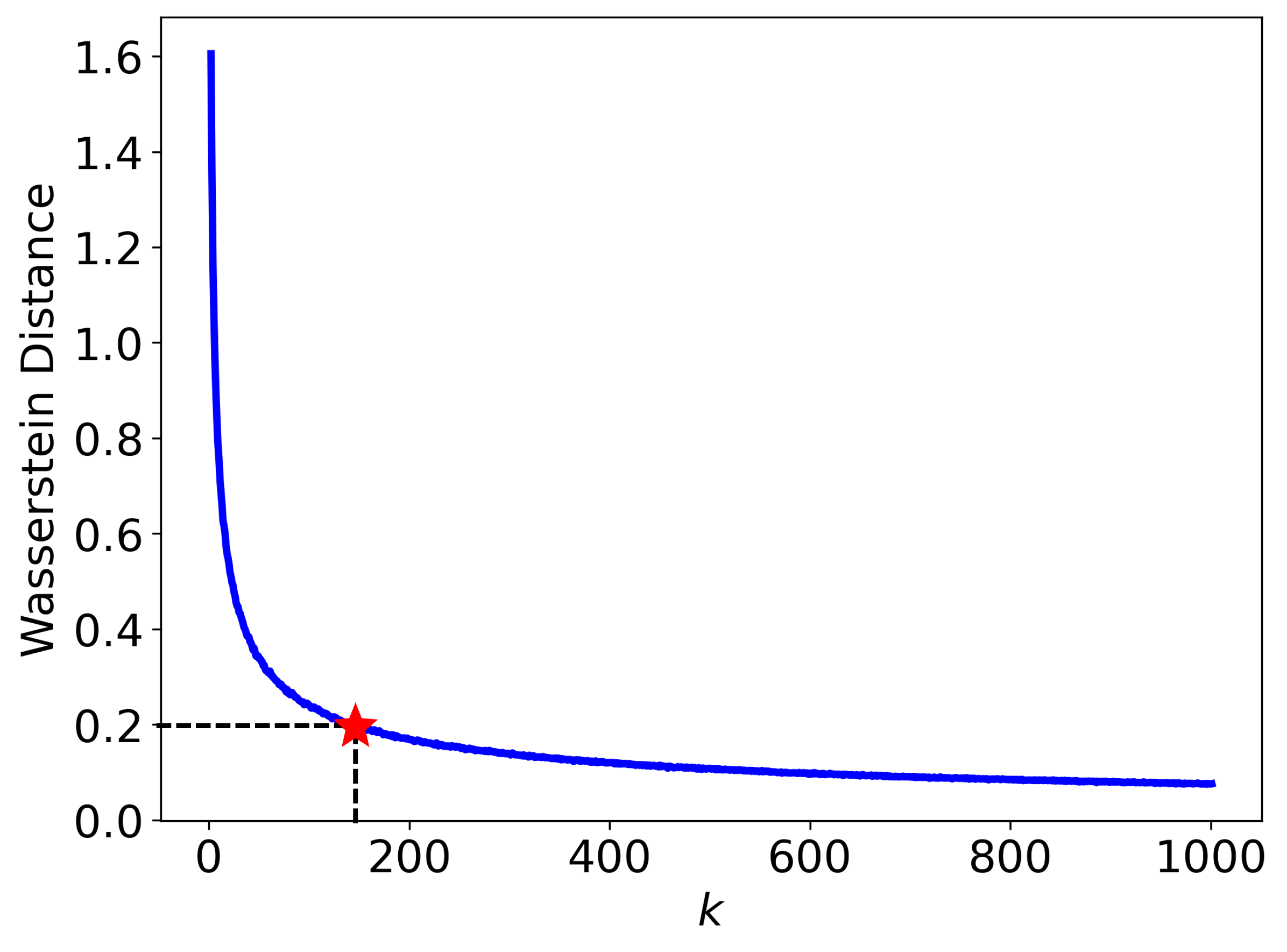}
    }
    \caption{Approximation of deterministic lockout time using Erlang-$k$ distribution.}
    \label{fig:k-erlang}
\end{figure}
\vspace{-6pt}

\subsection{Global Stability of the System}
Given $k$, this subsection further examines whether a sufficient condition that guarantees the global stability of the system holds as Proposition~\ref{global}. Specifically, we employ CVXPY, an open-source Python library for convex optimization \citep{cvx}, to examine global stability under various parameters, including transition rate, the fraction of permanent HDVs, and lockout time. We begin by investigating the impact of different transition rates on system stability, as shown in Figure~\ref{fig:global stable-lambda}(a–c). 

\begin{figure}[!h]
    \vspace{-6pt}
    \centering
    \setlength{\belowcaptionskip}{-8pt}
    \vspace{-12pt}
    \subfloat[$\lambda_4=0.01$]{%
        \includegraphics[height=0.17\textheight]{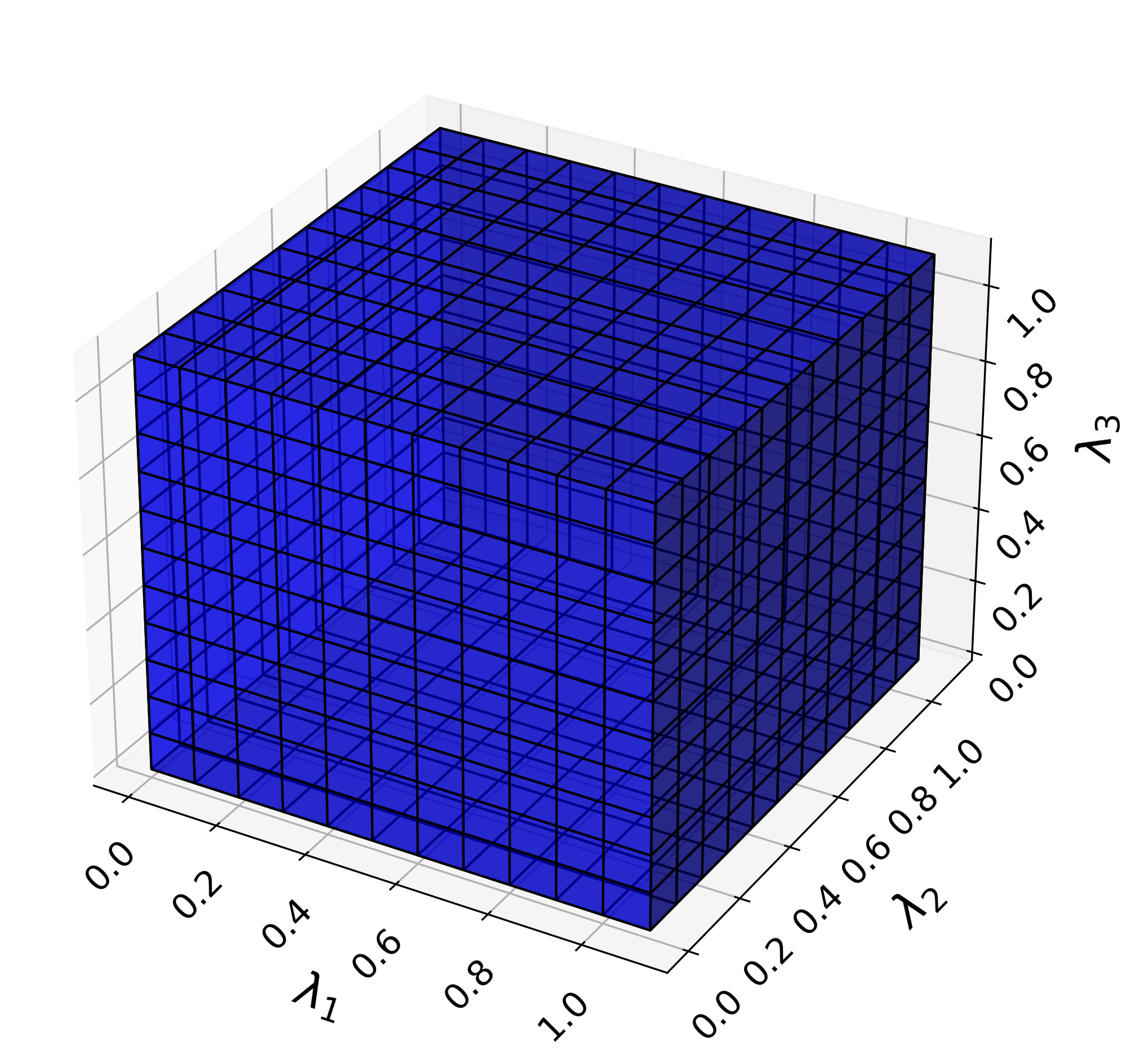}
    }
    \subfloat[$\lambda_4=0.1$]{%
        \includegraphics[height=0.17\textheight]{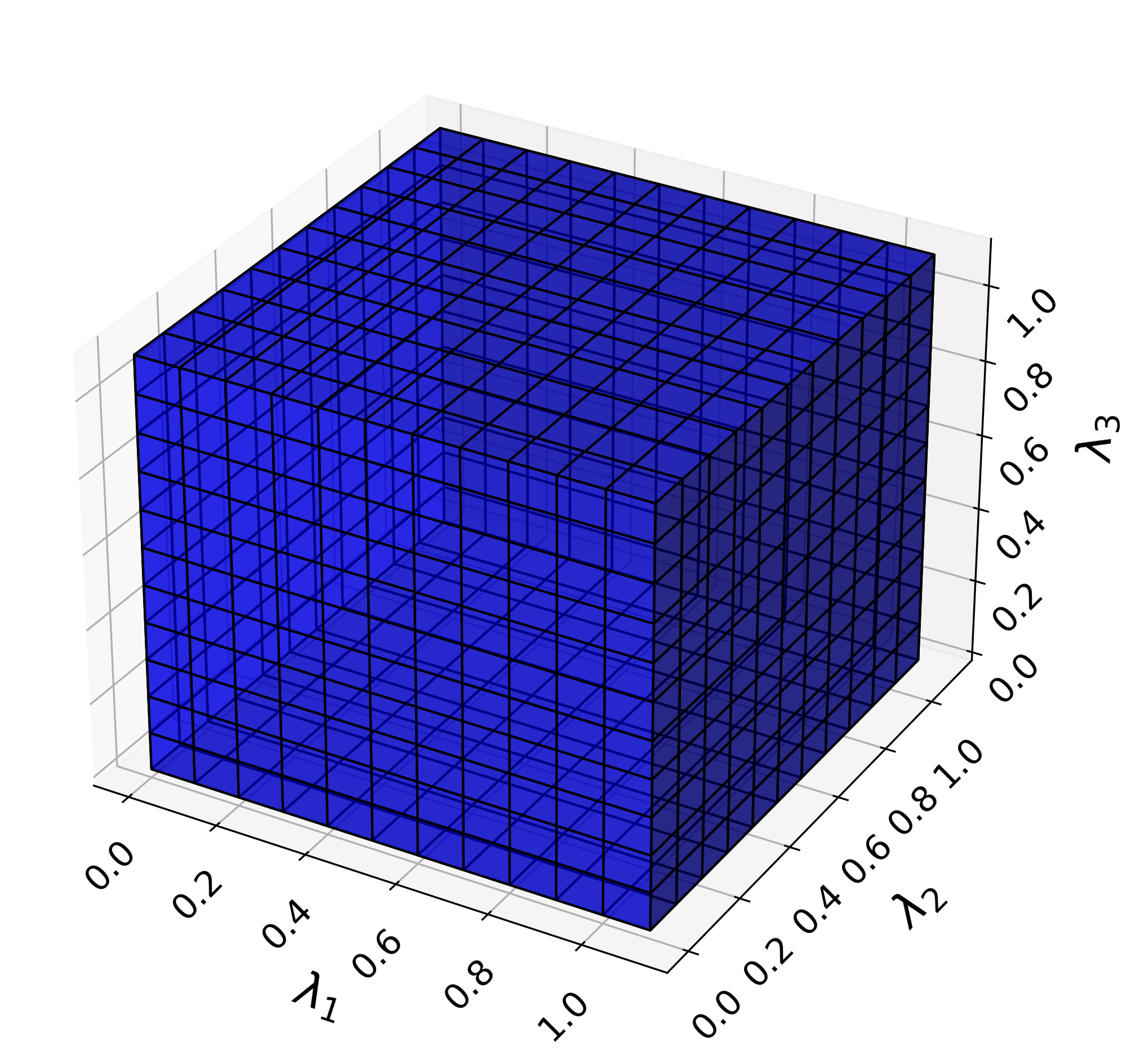}
    }
    \subfloat[$\lambda_4=0.5$]{%
        \includegraphics[height=0.17\textheight]{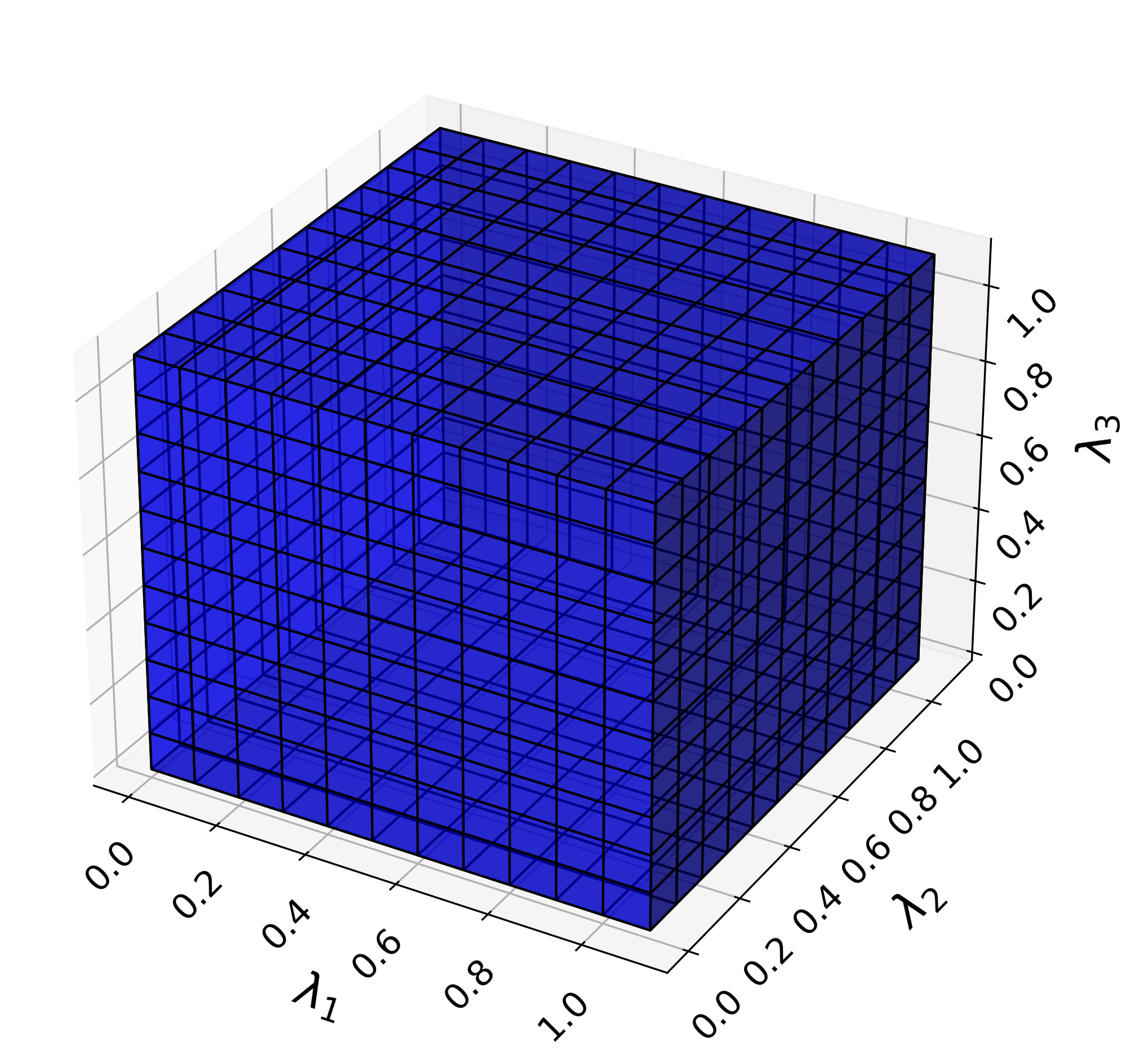}
    }\\[-14pt]
    \subfloat[$\lambda_1=0.1$, $\lambda_2=0.15$\\ $\lambda_3=0.9$, $\lambda_4=0.05$]{%
        \includegraphics[height=0.17\textheight]{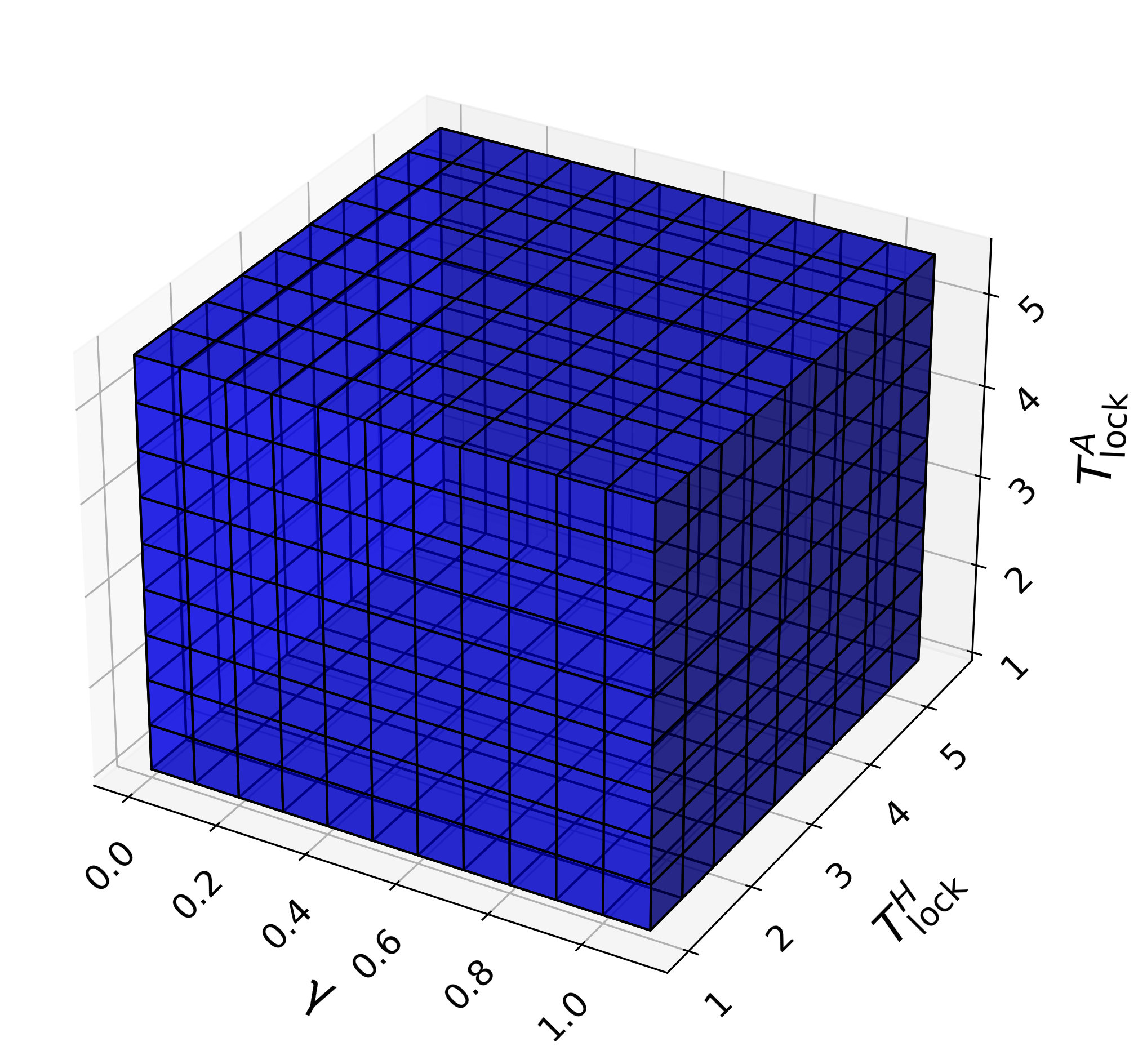}
    }
    \subfloat[$\lambda_1=0.05$, $\lambda_2=0.45$\\ $\lambda_3=0.65$, $\lambda_4=0.05$]{%
        \includegraphics[height=0.17\textheight]{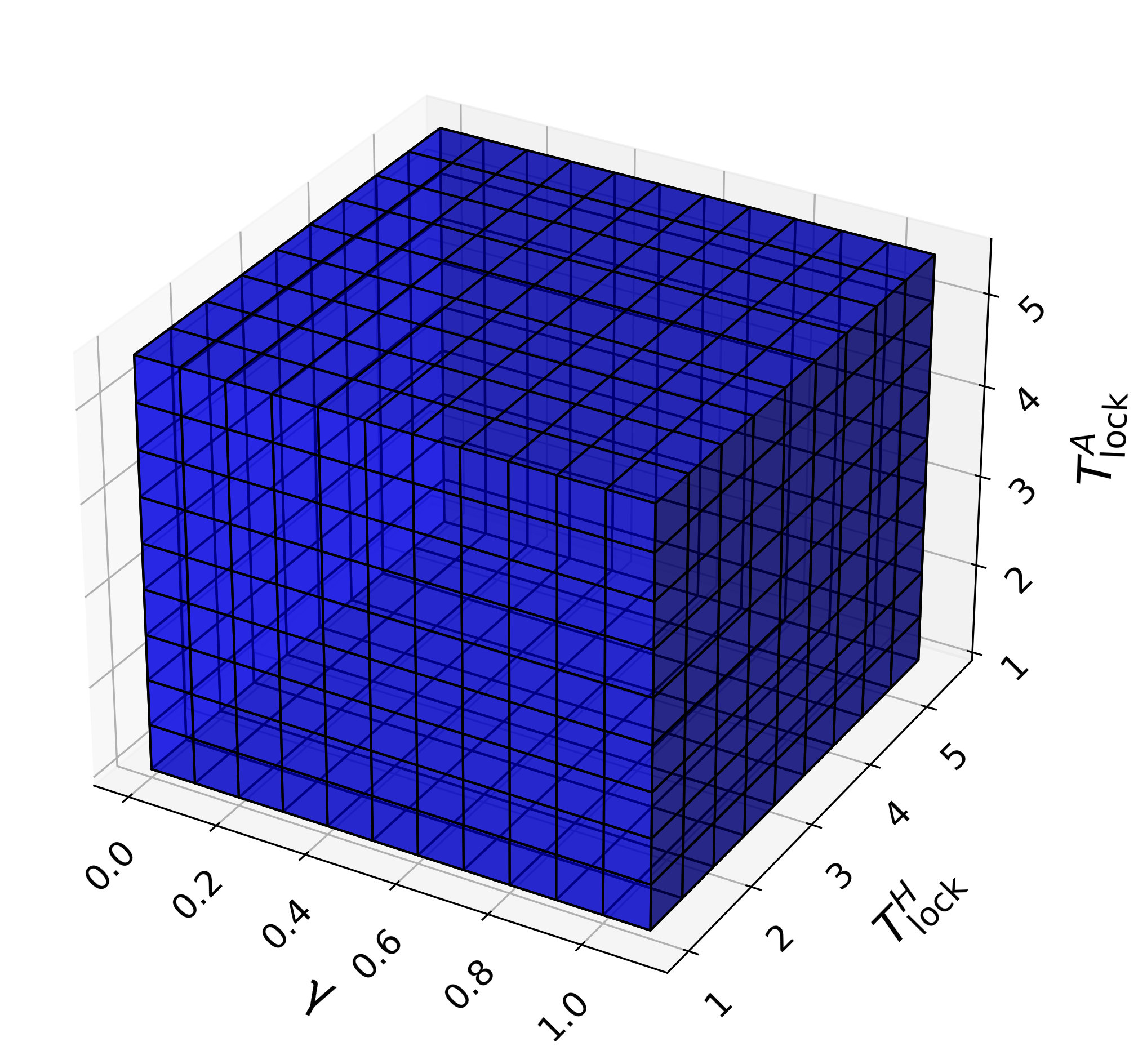}
    }
    \subfloat[$\lambda_1=0.05$, $\lambda_2=0.9$\\ $\lambda_3=0.15$, $\lambda_4=0.1$]{%
        \includegraphics[height=0.17\textheight]{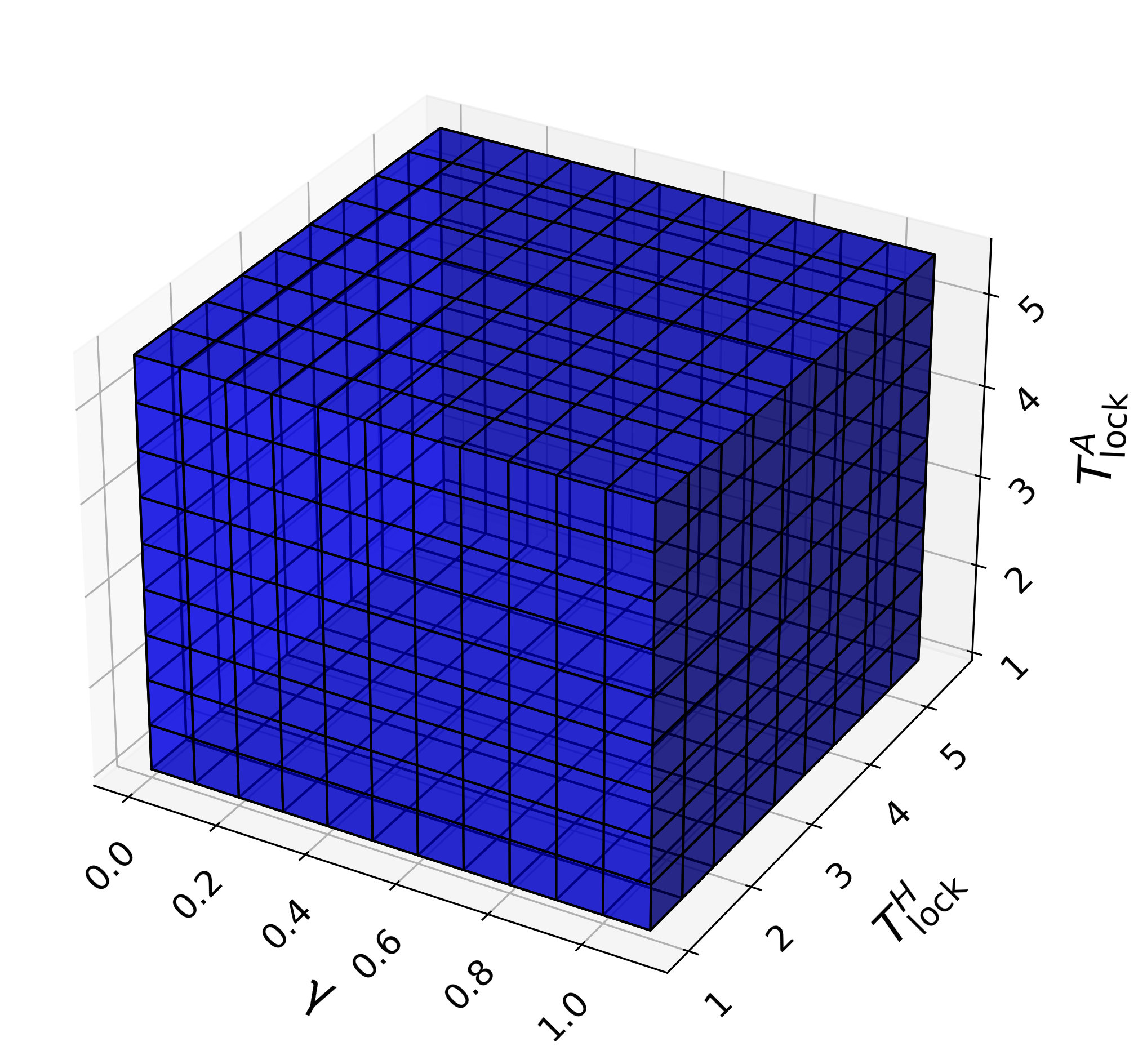}
    }
    \caption{Global stability regions under different transition rates (subfigures (a), (b), and (c) are plotted with \(\gamma = 0.2\), \(T^{H}_{\text{lock}} = 3\,\text{s}\), and \(T^{A}_{\text{lock}} = 3\,\text{s}\), while (d), (e), and (f) illustrate results under different traffic condition)}
    \label{fig:global stable-lambda}
\end{figure}
\vspace{-2pt}

The blue regions indicate the existence of a common positive definite matrix, confirming that the system is globally stable and admits a unique equilibrium regardless of the initial states. As mentioned in Assumption~1 in Section~\ref{assumption}, different transition rates correspond to different traffic conditions. Therefore, the results suggest that the system remains globally stable across a wide range of traffic conditions.

Similar results are observed under different permanent HDV rates and lockout times, further indicating that the system’s global stability holds across a range of parameter settings (Figure~\ref{fig:global stable-lambda}(d–f)). Although the generator matrix grows in dimension due to the semi-Markov to Markov approximation, making theoretical analysis more challenging, our numerical results demonstrate that the system remains globally stable despite its time-varying and leader-dependent structure.

\subsection{Leader-Dependent Dynamics}
\label{leader_dependent}

In this section, we aim to compare the system dynamics under leader-dependent and leader-independent mode transitions to better understand how the presence or absence of leader influence affects the evolution of mode transitions and resulting traffic throughput. Due to the limited availability of real-world datasets, calibrating transition rates remains challenging. Therefore, the transition rates used here are not based on empirical data but are selected to enable a consistent comparison between the two systems under identical settings. To highlight the differences in system dynamics, we compare two representative scenarios: one dominated by downward transitions (from AV to HDV mode) resembling congested or oscillation traffic, and the other dominated by upward transitions (from HDV to AV mode) reflecting free-flow conditions. For the leader-independent system, we set \(\lambda_{H \rightarrow A} = 0.1\), \(\lambda_{A \rightarrow H} = 0.5\) in the first scenario, and \(\lambda_{H \rightarrow A} = 0.5\), \(\lambda_{A \rightarrow H} = 0.1\) in the second. To ensure consistency, the leader-dependent system is configured such that the average of \(\lambda_1\) and \(\lambda_3\) equals \(\lambda_{H \rightarrow A}\), and the average of \(\lambda_2\) and \(\lambda_4\) equals \(\lambda_{A \rightarrow H}\). In both scenarios, the average traffic speed is fixed at \(10\,\text{m/s}\).

Figure~\ref{fig:traffic_down} compares the throughput dynamics of various leader-dependent transition scenarios with a leader-independent baseline under downward transition dominance. Among the leader-dependent settings, the blue curve shows the most pronounced and sustained drop in throughput. This configuration strongly favors downgrades to HDV mode, especially when the leader is an HDV, resulting in a steady-state throughput lower than that of the leader-independent system.  This result reflects a strong cascading effect: when more PAVs switch to HDV mode, they are more likely to become leaders for other vehicles. Since HDV leaders make their followers more likely to downgrade as well, this creates a feedback loop where downgrades keep spreading through the traffic. As a result, most PAVs end up staying in HDV mode, and the overall throughput remains low.

\begin{figure}[!ht]\centering
    \includegraphics[width=1\textwidth]{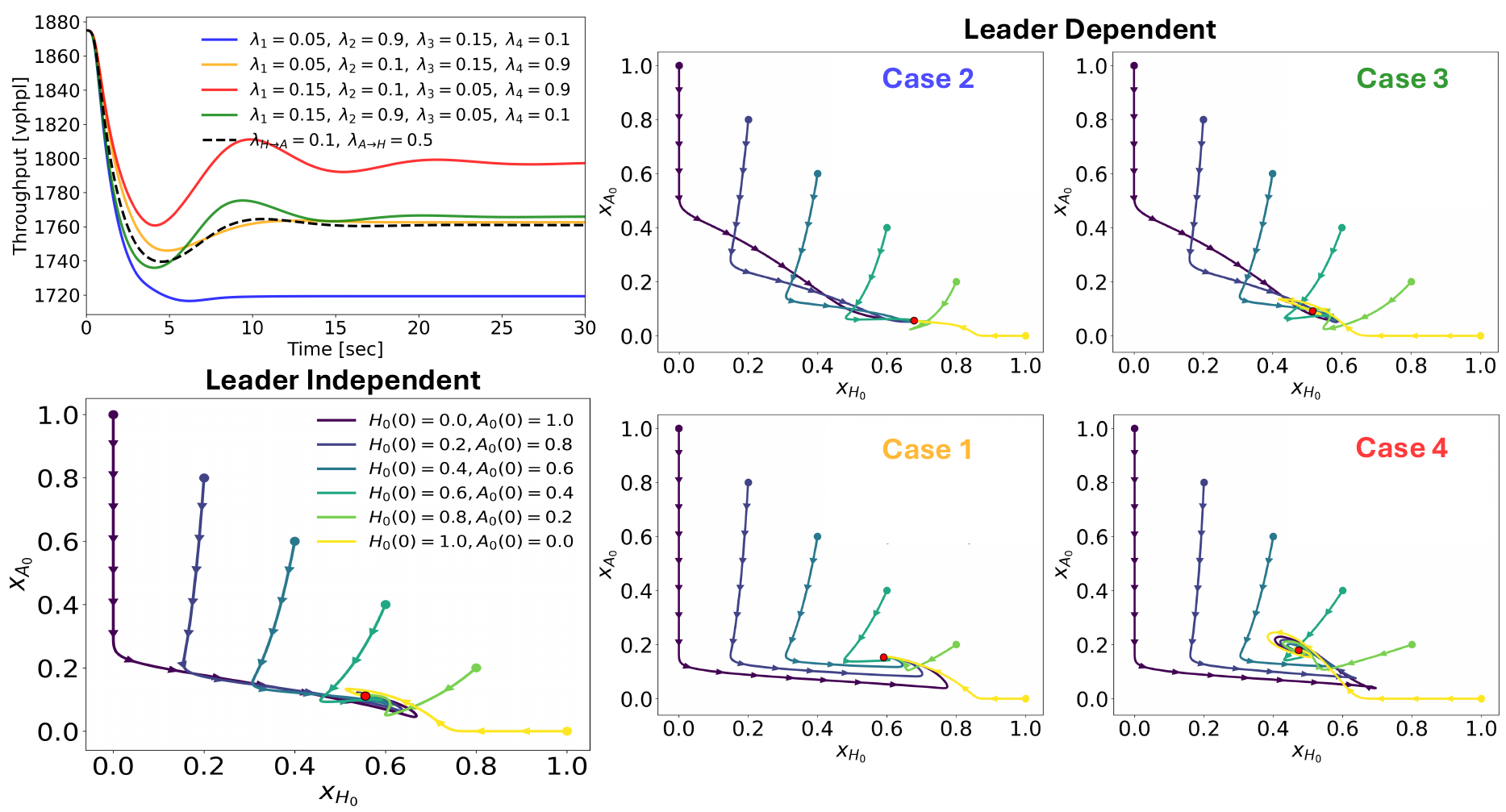}
    \vspace{-12pt} 
    \caption{Throughput evolution and phase diagram of leader-independent vs. leader-dependent systems under downward transition dominance (the case number colors correspond to the line colors shown in the upper-left subfigure)}
    \label{fig:traffic_down}
\end{figure}
\vspace{-6pt}

The red curve illustrates a scenario where PAVs with HDV leaders slightly prefer to transition upward to AV mode, while PAVs with AV leaders strongly tend to transition downward. This asymmetric transition pattern results in larger oscillations during the transient phase. However, because HDV-led PAVs retain a mild upward transition tendency, the cascading downgrade effect is mitigated over time, allowing more vehicles to return to AV mode. As a result, the system achieves a higher steady-state throughput than the leader-independent baseline. Nonetheless, such asymmetric transition behavior is uncommon in downward transition-dominated conditions, where strong downgrade tendencies typically prevail. The green and yellow curves both closely mirror the leader-independent case. Although they are technically leader-dependent, their transition parameters inherently bias them toward downgrading to HDV mode—regardless of the leader type. As a result, the system settles into a similar steady state with a high proportion of PAVs operating in HDV mode, and the leader identity plays a negligible role in shaping the overall dynamics. 

The phase diagrams illustrate that, across both leader independent and leader dependent systems, and even within the leader dependent system under different transition parameters, trajectories originating from different initial states consistently converge to a unique equilibrium, further confirming global stability. However, the convergence paths differ significantly, highlighting that transient dynamics are sensitive to both the transition structure and system parameters.

Under the upward dominance scenario (free flow condition, Figure~\ref{fig:traffic_up}), a cascading effect occurs when PAVs follow leaders of the dominant mode and prefer to transition in the same direction, as seen in the blue curve where \(\lambda_3 = 0.9\) encourages AV-led PAVs to remain in or switch to AV mode. This alignment accelerates convergence toward an AV-dominated equilibrium. In contrast, the red curve illustrates an asymmetric transition pattern: HDV-led PAVs prefer to switch to AV mode, while AV-led PAVs slightly prefer to downgrade. This mismatch disrupts coordinated transitions and results in larger fluctuations before the system stabilizes. The yellow and green curves follow a consistent pattern in which PAVs prefer upward transitions regardless of leader type. Consequently, both exhibit similar dynamics to the leader-independent case under upward dominance, with smooth convergence and limited fluctuation.

\begin{figure}[!ht]\centering
    \includegraphics[width=1\textwidth]{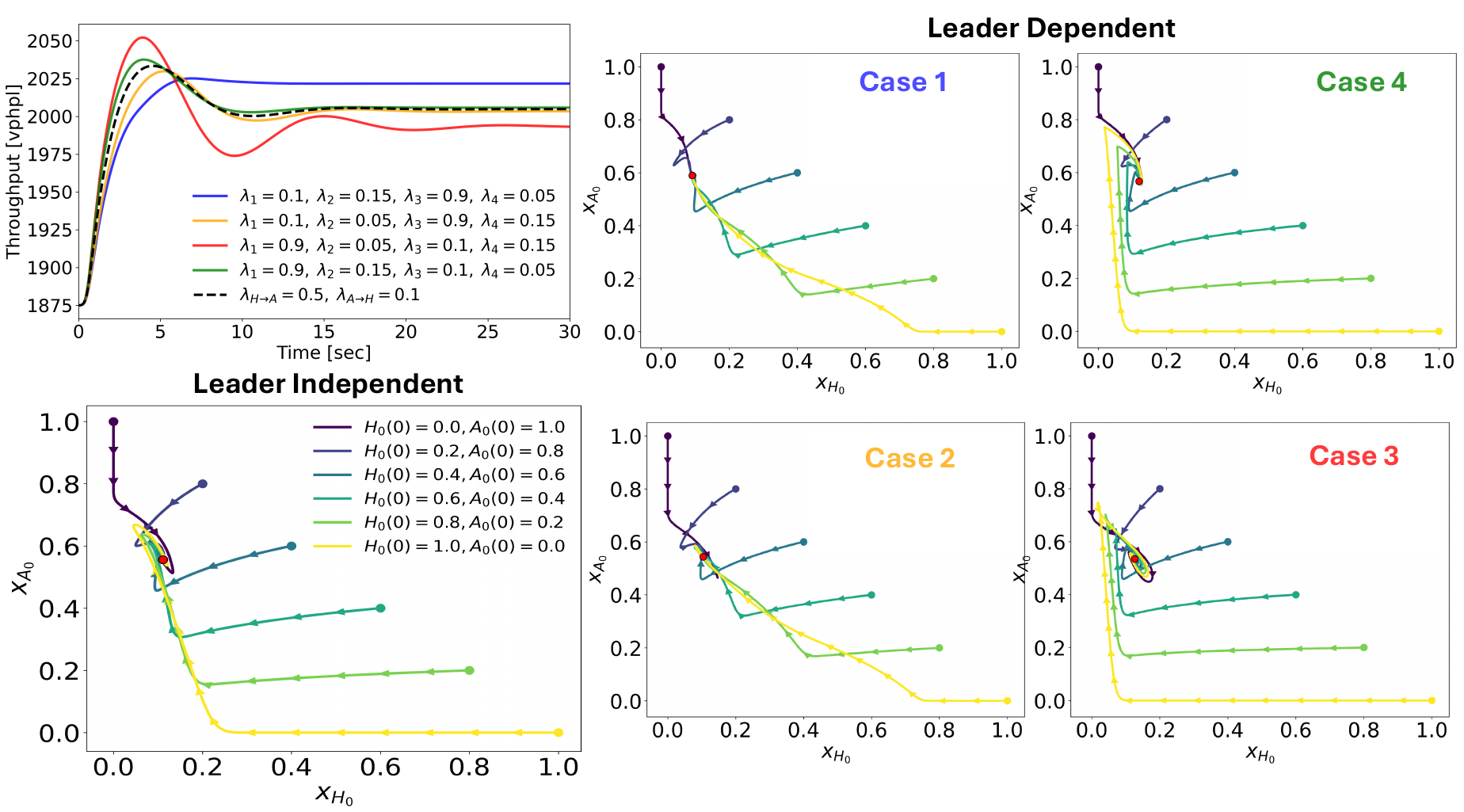}
    \vspace{-12pt} 
    \caption{Throughput evolution and phase diagram of leader-independent vs. leader-dependent systems under upward transition dominance}
    \label{fig:traffic_up}
\end{figure}
\vspace{-6pt}

Compared to the downward transition scenario, the difference in steady-state throughput between the cascading effect and the leader-independent baseline is noticeably smaller under upward transition dominance. This outcome can be primarily attributed to a ceiling effect. Even when most PAVs operate in AV mode, the presence of permanent HDVs ($\gamma=0.2$) imposes a hard limit on the maximum achievable throughput because these vehicles maintain longer headways. As a result, while cascading upward transitions improve efficiency, the benefit is inherently capped. In contrast, cascading downward transitions cause a rapid increase in PAVs operating in HDV mode, which significantly degrades system efficiency and amplifies the performance difference relative to the leader independent system

\subsection{Throughput Analysis Using NGSIM Data}

As discussed in Section~\ref{leader_dependent}, the cascading effect is more pronounced under downward transition dominance. To examine this phenomenon under real-world conditions, we apply our model to the NGSIM dataset \citep{ngsimdatasetus101, ngsimdatasetus80}, which provides empirical traffic speed profiles (Figure \ref{NGSIM_speed}). Specifically, the speed trajectory from the dataset is treated as the time-varying average traffic speed \(v(t)\), and is then integrated into the throughput analysis. This setup allows us to evaluate how realistic speed fluctuations influence capacity, given that throughput is closely linked to prevailing traffic speed.

\begin{figure}[!ht]\centering
    \includegraphics[width=0.75\textwidth]{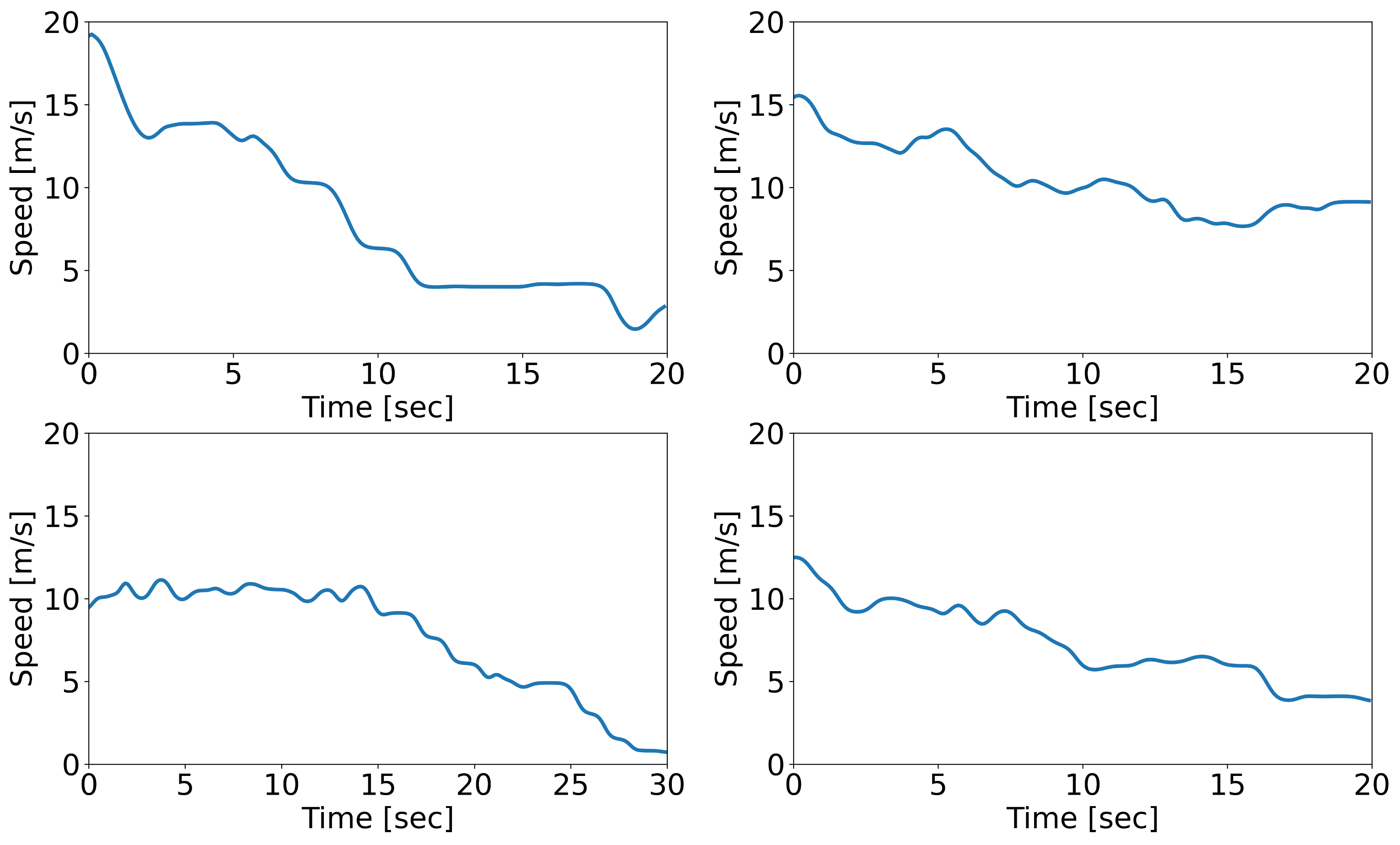}
    \vspace{-12pt} 
    \caption{Selected speed profile in the NGSIM dataset}
    \label{NGSIM_speed}
\end{figure}
\vspace{-6pt}

In Figure \ref{NGSIM_cap}, we find that, under these empirical conditions, the impact of mode transitions becomes even more significant, with the maximum throughput difference reaching approximately 100~vphpl. Notably, transition settings that induce cascading effects (blue line) consistently result in lower throughput throughout the time horizon.

\begin{figure}[!ht]\centering
    \includegraphics[width=1\textwidth]{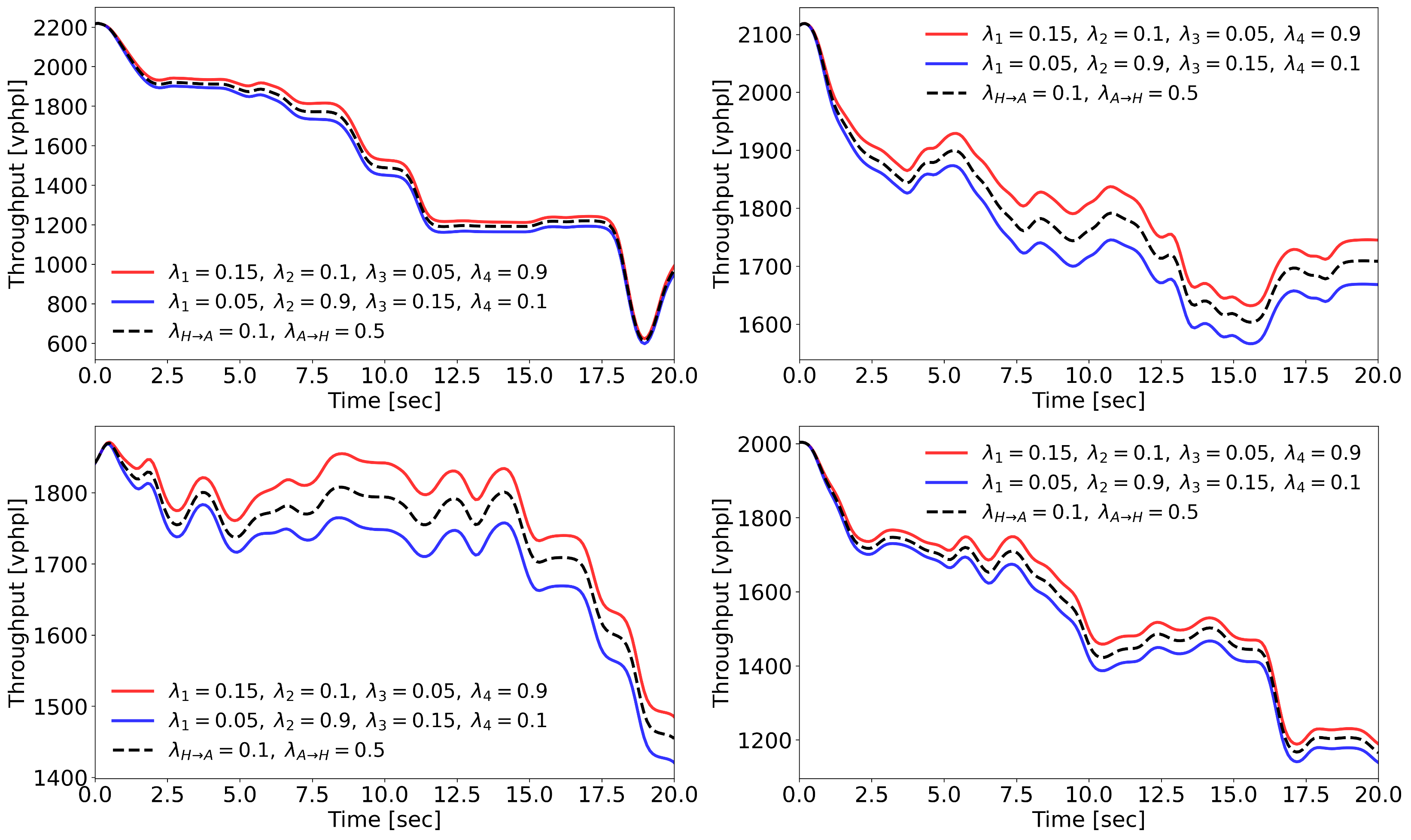}
    \vspace{-12pt} 
    \caption{Throughput evolution under empirical speed profiles from the NGSIM dataset }
    \label{NGSIM_cap}
\end{figure}
\vspace{-6pt}
\newpage

\subsection{Sensitivity Analysis}
To gain a more comprehensive understanding of the leader-dependent system, we conduct sensitivity analyses on key parameters including the permanent HDV rate and the initial fraction of AV and HDV modes among PAVs.

\subsubsection{Permanent HDV Rate ($\gamma$)}
Throughput fluctuation is quantified by the \(L_2\) norm of the deviation from steady‑state throughput. Larger values indicate more pronounced fluctuation over time. For a fair comparison with the leader‑independent system, we set the average of \(\lambda_1\) and \(\lambda_3\) equal to \(\lambda_{H\rightarrow A}\) and the average of \(\lambda_2\) and \(\lambda_4\) equal to \(\lambda_{A\rightarrow H}\). Once \(\lambda_1\) and \(\lambda_2\) are chosen, \(\lambda_3\) and \(\lambda_4\) are determined accordingly.  In the diverging colour maps, the leader-independent baseline is set as the transition point in the color scale.

Under downward transition dominance (Figure~\ref{fig:HM_downward}), we observe consistent behavioral patterns across different levels of permanent HDV penetration. The leader-dependent cases in the top-left and bottom-right corners of the throughput and fluctuation heatmaps correspond to the blue and red curves in Figure~\ref{fig:traffic_down}, respectively. These settings exemplify two key phenomena identified earlier: cascading downgrades and strong transitional fluctuations due to asymmetric preferences. Across all values of \(\gamma\), the cascading effect persists—when transition preferences align with the dominant HDV mode, PAVs are increasingly drawn into HDV mode. This alignment leads to a gradual but sustained decrease in throughput. However, the severity of throughput fluctuation declines as \(\gamma\) increases. For example, the throughput flcutation range of throughput decreases from 120~vphpl to 92~vphpl as \(\gamma\) increases from 0.2 to 0.5, and then to 47~vphpl, when \(\gamma\) reaches 0.8. This trend occurs because higher permanent HDV rates reduce the fraction of PAVs, thereby naturally dampening the extent of throughput fluctuation and weakening the cascading dynamics.

\begin{figure}[!ht]\centering
    \includegraphics[width=1\textwidth]{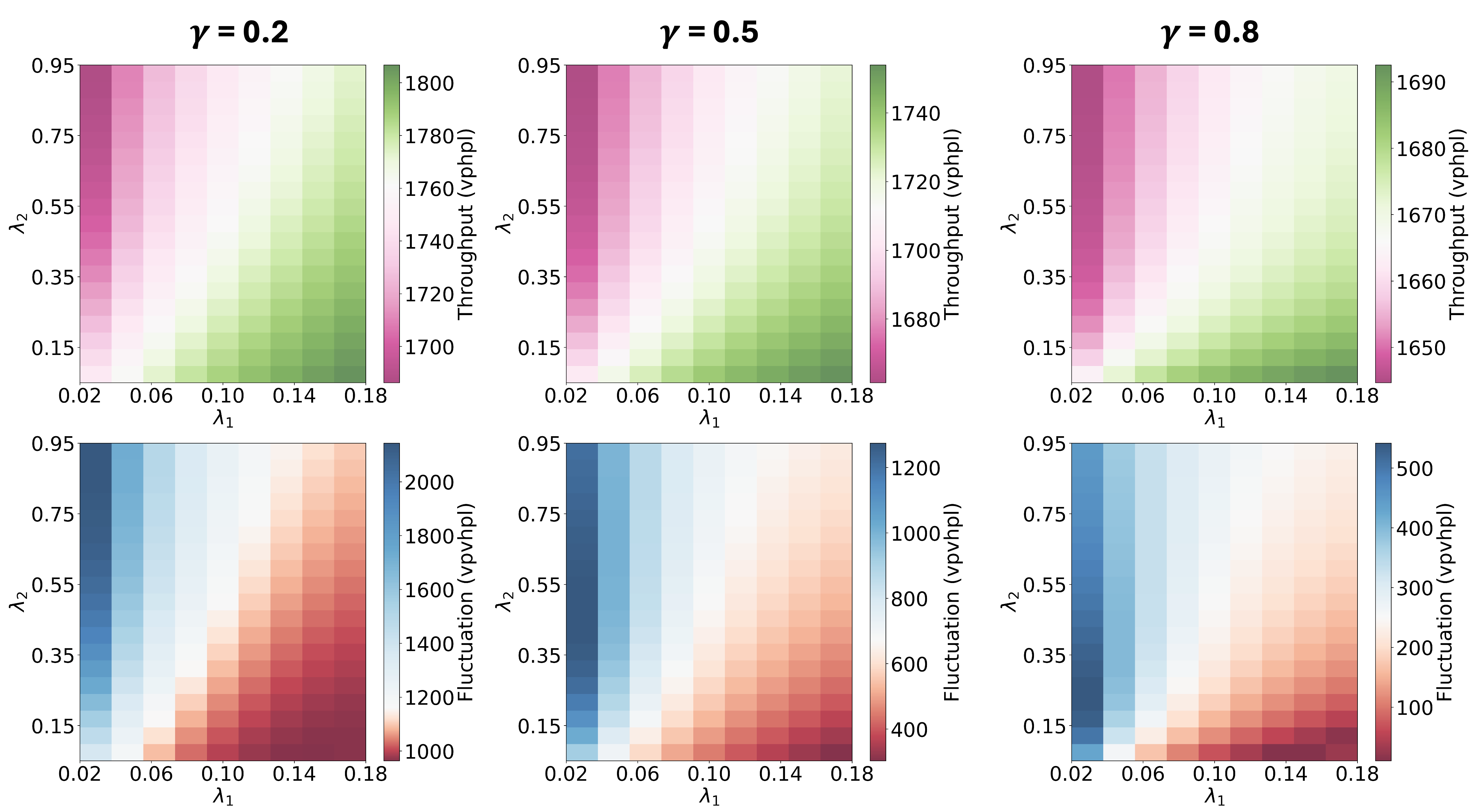}
    \vspace{-12pt} 
    \caption{Throughput and fluctuation under downward transition dominance ($\bar{\lambda}_{H \rightarrow A} = 0.1$, $\bar{\lambda}_{A \rightarrow H} = 0.5$) under different permanent HDV rate}
    \label{fig:HM_downward}
\end{figure}
\vspace{-6pt}

\subsubsection{Initial fraction of AV and HDV modes among PAVs}

Under the cascading downgrade scenario (Figure~\ref{fig:Inital}(a)), although the cascading effect leads to a lower steady-state throughput, it also results in faster convergence to equilibrium compared to the leader-independent case (Figure~\ref{fig:Inital}(c)). Naturally, the greater the deviation of the initial state from equilibrium, the more pronounced the fluctuations become, leading to longer convergence times.

In contrast, the scenario characterized by strong fluctuations due to asymmetric transition preferences (Figure~\ref{fig:Inital}(b)) exhibits persistent overshoots, regardless of the initial state. The severity of these fluctuations can reach up to 123~vphpl and is generally more intense than in the leader-independent baseline  (Figure~\ref{fig:Inital}(d)) . These findings highlight that such asymmetry in transition behavior can significantly amplify traffic instability.

These dynamic behaviors underscore the critical role of the initial PAV mode fraction. When the traffic state undergoes abrupt changes---such as a shift from free-flow to congestion, which frequently occurs in merging or diverging highway sections---the system is more likely to exhibit intensified transitional responses. For instance, if the initial composition is dominated by AV-mode PAVs and congestion arises, transition preferences may rapidly shift toward HDV mode. This shift creates a large initial deviation from the new equilibrium, intensifying the throughput fluctuations. Therefore, traffic fluctuations arise not only from the transition rate settings but also from the initial fraction of AV and HDV modes between PAVs, particularly in response to sudden changes in traffic conditions.

\begin{figure}[!ht]\centering
    \includegraphics[width=1\textwidth]{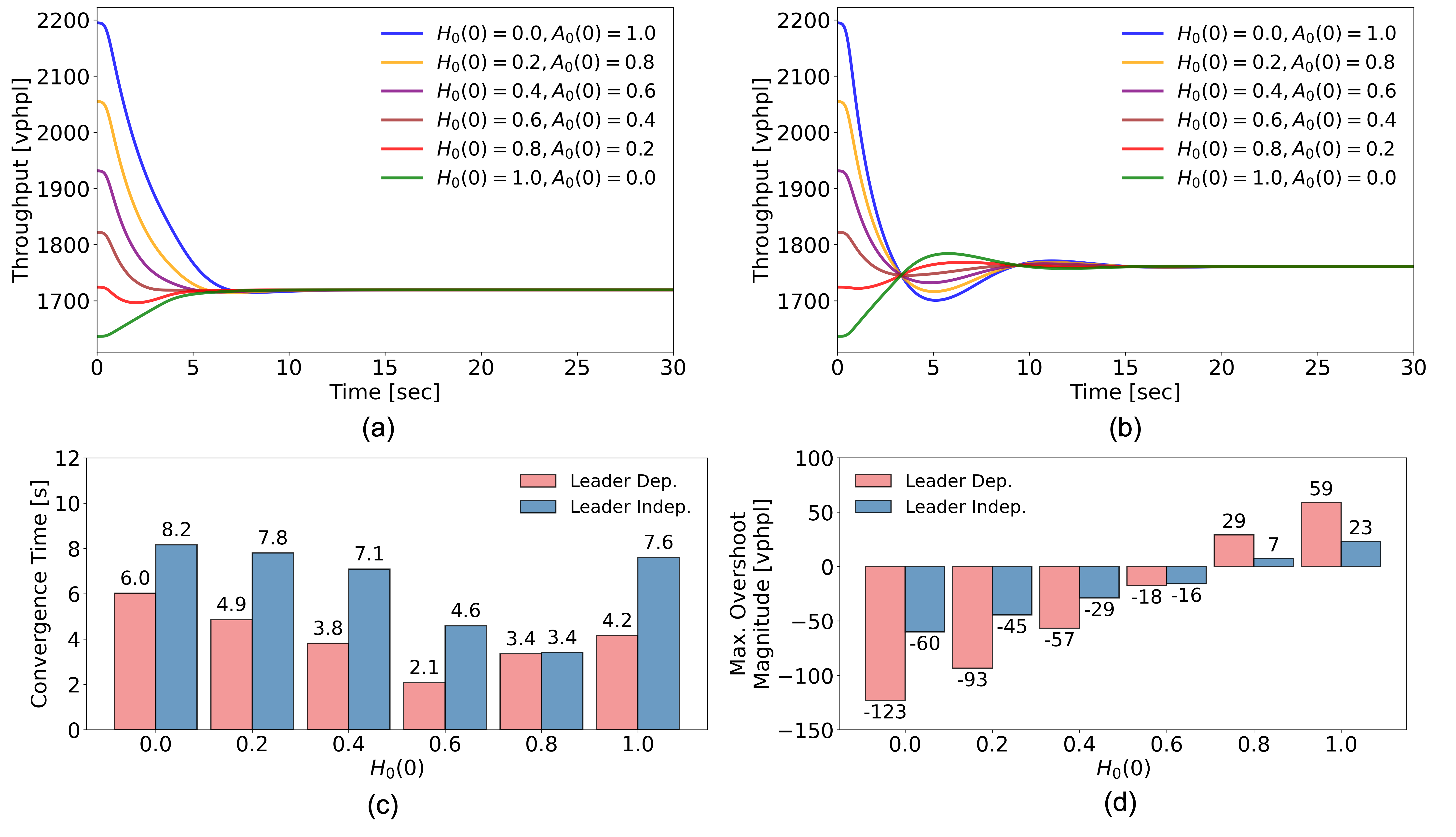}
    \vspace{-12pt} 
    \caption{Throughput evolution under different transition rate settings under downward transition dominance, respectively; (c) presents convergence time for (a); and (d) shows overshoot magnitude for (b))}
    \label{fig:Inital}
\end{figure}
\vspace{-6pt}
\newpage

\section{Conclusion}
\vspace{-6pt}

This paper presents a novel traffic throughput modeling framework for partially automated traffic that captures cascading driver interventions through a leader-dependent mode-switching mechanism. Unlike traditional mixed traffic models, our framework systematically considers transitions between AVs and HDVs—including the associated transitional behaviors and lockout processes—making it applicable to vehicles with varying automation levels and different AV market penetration rates. To manage the system's complexity, we introduce a continuous semi-Markov chain with lockout constraints, further formulating it as a nonlinear dynamical system with a phase-type distribution approximation. Consequently, the transitional throughput can be effectively solved using standard nonlinear system solvers such as the Runge-Kutta algorithm. Moreover, we demonstrate that the system can be viewed as a piecewise affine, state-dependent linear model, and we rigorously prove the existence of an equilibrium state, as well as global stability, using Brouwer's Fixed Point Theorem,  and the 1D Uncertainty Polytopes Theorem.

Numerical experiments validate the theoretical findings, showing that leader-dependent mode transitions help explain cascading impacts in mixed traffic. These cascading effects occur when PAVs tend to follow the mode of their leaders, especially when their leader matches the dominant traffic mode. This leads to widespread transitions in the same direction. On the other hand, strong fluctuations in throughput happen when PAVs react differently depending on whether they follow AV or HDV leaders. This mismatch disrupts coordination and increases instability. Sensitivity analysis shows that the permanent HDV rate does not change these patterns but affects how much throughput varies. When there are more permanent HDVs, the fluctuation is reduced because fewer PAVs can switch modes. The initial distribution of PAV modes also matters—larger differences from the steady state lead to stronger transitional responses, especially when traffic conditions change suddenly.

Future work will focus on validating the proposed framework using real-world traffic data from partially automated vehicles. This includes analyzing empirical mode-switching behaviors and refining model parameters to enhance accuracy. Additionally, integrating more complex driver decision-making factors and testing the framework in diverse traffic scenarios will provide deeper insights into real-world application.

\section*{Acknowledgment}
\vspace{-6pt}
 This work is supported by the US National Science Foundation (Award CMMI 2129765).

 \bibliographystyle{elsarticle-harv} 
 \bibliography{cas-refs}

\begin{thebibliography}{53}
\expandafter\ifx\csname natexlab\endcsname\relax\def\natexlab#1{#1}\fi
\providecommand{\url}[1]{\texttt{#1}}
\providecommand{\href}[2]{#2}
\providecommand{\path}[1]{#1}
\providecommand{\DOIprefix}{doi:}
\providecommand{\ArXivprefix}{arXiv:}
\providecommand{\URLprefix}{URL: }
\providecommand{\Pubmedprefix}{pmid:}
\providecommand{\doi}[1]{\href{http://dx.doi.org/#1}{\path{#1}}}
\providecommand{\Pubmed}[1]{\href{pmid:#1}{\path{#1}}}
\providecommand{\bibinfo}[2]{#2}
\ifx\xfnm\relax \def\xfnm[#1]{\unskip,\space#1}\fi
\bibitem[{Ahn and Cassidy(2007)}]{ahn2007freeway}
\bibinfo{author}{Ahn, S.}, \bibinfo{author}{Cassidy, M.J.}, \bibinfo{year}{2007}.
\newblock \bibinfo{title}{Freeway traffic oscillations and vehicle lane-change maneuvers}.
\newblock \bibinfo{journal}{Transportation and Traffic Theory} \bibinfo{volume}{1}, \bibinfo{pages}{691--710}.
\bibitem[{Anderson(2012)}]{anderson2012continuous}
\bibinfo{author}{Anderson, W.J.}, \bibinfo{year}{2012}.
\newblock \bibinfo{title}{Continuous-time Markov chains: An applications-oriented approach}.
\newblock \bibinfo{publisher}{Springer Science \& Business Media}.
\bibitem[{Arjovsky et~al.(2017)Arjovsky, Chintala and Bottou}]{arjovsky2017wasserstein}
\bibinfo{author}{Arjovsky, M.}, \bibinfo{author}{Chintala, S.}, \bibinfo{author}{Bottou, L.}, \bibinfo{year}{2017}.
\newblock \bibinfo{title}{Wasserstein generative adversarial networks}, in: \bibinfo{booktitle}{International conference on machine learning}, \bibinfo{organization}{PMLR}. pp. \bibinfo{pages}{214--223}.
\bibitem[{Bellem et~al.(2018)Bellem, Thiel, Schrauf and Krems}]{bellem2018comfort}
\bibinfo{author}{Bellem, H.}, \bibinfo{author}{Thiel, B.}, \bibinfo{author}{Schrauf, M.}, \bibinfo{author}{Krems, J.F.}, \bibinfo{year}{2018}.
\newblock \bibinfo{title}{Comfort in automated driving: An analysis of preferences for different automated driving styles and their dependence on personality traits}.
\newblock \bibinfo{journal}{Transportation research part F: traffic psychology and behaviour} \bibinfo{volume}{55}, \bibinfo{pages}{90--100}.
\bibitem[{Bladt(2005)}]{bladt2005review}
\bibinfo{author}{Bladt, M.}, \bibinfo{year}{2005}.
\newblock \bibinfo{title}{A review on phase-type distributions and their use in risk theory}.
\newblock \bibinfo{journal}{ASTIN Bulletin: The Journal of the IAA} \bibinfo{volume}{35}, \bibinfo{pages}{145--161}.
\bibitem[{Brouwer(1911)}]{brouwer1911abbildung}
\bibinfo{author}{Brouwer, L.E.J.}, \bibinfo{year}{1911}.
\newblock \bibinfo{title}{{\"U}ber abbildung von mannigfaltigkeiten}.
\newblock \bibinfo{journal}{Mathematische annalen} \bibinfo{volume}{71}, \bibinfo{pages}{97--115}.
\bibitem[{Burden and Faires(1997)}]{burden1997numerical}
\bibinfo{author}{Burden, R.L.}, \bibinfo{author}{Faires, J.D.}, \bibinfo{year}{1997}.
\newblock \bibinfo{title}{Numerical analysis, brooks}.
\bibitem[{del Castillo(2001)}]{del2001propagation}
\bibinfo{author}{del Castillo, J.M.}, \bibinfo{year}{2001}.
\newblock \bibinfo{title}{Propagation of perturbations in dense traffic flow: a model and its implications}.
\newblock \bibinfo{journal}{Transportation Research Part B: Methodological} \bibinfo{volume}{35}, \bibinfo{pages}{367--389}.
\bibitem[{Chen et~al.(2017)Chen, Ahn, Chitturi and Noyce}]{chen2017towards}
\bibinfo{author}{Chen, D.}, \bibinfo{author}{Ahn, S.}, \bibinfo{author}{Chitturi, M.}, \bibinfo{author}{Noyce, D.A.}, \bibinfo{year}{2017}.
\newblock \bibinfo{title}{Towards vehicle automation: Roadway capacity formulation for traffic mixed with regular and automated vehicles}.
\newblock \bibinfo{journal}{Transportation research part B: methodological} \bibinfo{volume}{100}, \bibinfo{pages}{196--221}.
\bibitem[{Chen et~al.(2014)Chen, Ahn, Laval and Zheng}]{chen2014periodicity}
\bibinfo{author}{Chen, D.}, \bibinfo{author}{Ahn, S.}, \bibinfo{author}{Laval, J.}, \bibinfo{author}{Zheng, Z.}, \bibinfo{year}{2014}.
\newblock \bibinfo{title}{On the periodicity of traffic oscillations and capacity drop: The role of driver characteristics}.
\newblock \bibinfo{journal}{Transportation research part B: methodological} \bibinfo{volume}{59}, \bibinfo{pages}{117--136}.
\bibitem[{Chen et~al.(2012)Chen, Laval, Zheng and Ahn}]{chen2012behavioral}
\bibinfo{author}{Chen, D.}, \bibinfo{author}{Laval, J.}, \bibinfo{author}{Zheng, Z.}, \bibinfo{author}{Ahn, S.}, \bibinfo{year}{2012}.
\newblock \bibinfo{title}{A behavioral car-following model that captures traffic oscillations}.
\newblock \bibinfo{journal}{Transportation research part B: methodological} \bibinfo{volume}{46}, \bibinfo{pages}{744--761}.
\bibitem[{Chen et~al.(2023)Chen, Gong, Wang, Wang, Zhou and Ran}]{chen2023stochastic}
\bibinfo{author}{Chen, T.}, \bibinfo{author}{Gong, S.}, \bibinfo{author}{Wang, M.}, \bibinfo{author}{Wang, X.}, \bibinfo{author}{Zhou, Y.}, \bibinfo{author}{Ran, B.}, \bibinfo{year}{2023}.
\newblock \bibinfo{title}{Stochastic capacity analysis for a distributed connected automated vehicle virtual car-following control strategy}.
\newblock \bibinfo{journal}{Transportation research part C: emerging technologies} \bibinfo{volume}{152}, \bibinfo{pages}{104176}.
\bibitem[{Chen et~al.(2024)Chen, Zhang, Bai, Jiang, Ding and Wei}]{chen2024sigmoid}
\bibinfo{author}{Chen, X.}, \bibinfo{author}{Zhang, W.}, \bibinfo{author}{Bai, H.}, \bibinfo{author}{Jiang, R.}, \bibinfo{author}{Ding, H.}, \bibinfo{author}{Wei, L.}, \bibinfo{year}{2024}.
\newblock \bibinfo{title}{A sigmoid-based car-following model to improve acceleration stability in traffic oscillation and following failure in free flow}.
\newblock \bibinfo{journal}{IEEE Transactions on Intelligent Transportation Systems} .
\bibitem[{Colley et~al.(2022)Colley, Hummler and Rukzio}]{colley2022effects}
\bibinfo{author}{Colley, M.}, \bibinfo{author}{Hummler, C.}, \bibinfo{author}{Rukzio, E.}, \bibinfo{year}{2022}.
\newblock \bibinfo{title}{Effects of mode distinction, user visibility, and vehicle appearance on mode confusion when interacting with highly automated vehicles}.
\newblock \bibinfo{journal}{Transportation research part F: traffic psychology and behaviour} \bibinfo{volume}{89}, \bibinfo{pages}{303--316}.
\bibitem[{Colyar and Halkias(2006)}]{ngsimdatasetus80}
\bibinfo{author}{Colyar, J.}, \bibinfo{author}{Halkias, J.}, \bibinfo{year}{2006}.
\newblock \bibinfo{title}{{U.S. Highway 80 dataset}}.
\newblock \bibinfo{type}{Dataset}. Federal Highway Administration.
\newblock \bibinfo{note}{NGSIM}.
\bibitem[{Colyar and Halkias(2007)}]{ngsimdatasetus101}
\bibinfo{author}{Colyar, J.}, \bibinfo{author}{Halkias, J.}, \bibinfo{year}{2007}.
\newblock \bibinfo{title}{{U.S. Highway 101 dataset}}.
\newblock \bibinfo{type}{Dataset}. Federal Highway Administration.
\newblock \bibinfo{note}{NGSIM}.
\bibitem[{Correa et~al.(2018)Correa, Maerivoet, Mintsis, Wijbenga, Sepulcre, Rondinone, Schindler and Gozalvez}]{correa2018management}
\bibinfo{author}{Correa, A.}, \bibinfo{author}{Maerivoet, S.}, \bibinfo{author}{Mintsis, E.}, \bibinfo{author}{Wijbenga, A.}, \bibinfo{author}{Sepulcre, M.}, \bibinfo{author}{Rondinone, M.}, \bibinfo{author}{Schindler, J.}, \bibinfo{author}{Gozalvez, J.}, \bibinfo{year}{2018}.
\newblock \bibinfo{title}{Management of transitions of control in mixed traffic with automated vehicles}, in: \bibinfo{booktitle}{2018 16th International Conference on Intelligent Transportation Systems Telecommunications (ITST)}, \bibinfo{organization}{IEEE}. pp. \bibinfo{pages}{1--7}.
\bibitem[{CVX~Research(2012)}]{cvx}
\bibinfo{author}{CVX~Research, I.}, \bibinfo{year}{2012}.
\newblock \bibinfo{title}{{CVX}: Matlab software for disciplined convex programming, version 2.0}.
\newblock \bibinfo{howpublished}{\url{https://cvxr.com/cvx}}.
\bibitem[{David and Larry(1987)}]{david1987least}
\bibinfo{author}{David, A.}, \bibinfo{author}{Larry, S.}, \bibinfo{year}{1987}.
\newblock \bibinfo{title}{The least variable phase type distribution is erlang}.
\newblock \bibinfo{journal}{Stochastic Models} \bibinfo{volume}{3}, \bibinfo{pages}{467--473}.
\bibitem[{Eriksson and Stanton(2017)}]{eriksson2017takeover}
\bibinfo{author}{Eriksson, A.}, \bibinfo{author}{Stanton, N.A.}, \bibinfo{year}{2017}.
\newblock \bibinfo{title}{Takeover time in highly automated vehicles: noncritical transitions to and from manual control}.
\newblock \bibinfo{journal}{Human factors} \bibinfo{volume}{59}, \bibinfo{pages}{689--705}.
\bibitem[{Gershon et~al.(2021)Gershon, Seaman, Mehler, Reimer and Coughlin}]{gershon2021driver}
\bibinfo{author}{Gershon, P.}, \bibinfo{author}{Seaman, S.}, \bibinfo{author}{Mehler, B.}, \bibinfo{author}{Reimer, B.}, \bibinfo{author}{Coughlin, J.}, \bibinfo{year}{2021}.
\newblock \bibinfo{title}{Driver behavior and the use of automation in real-world driving}.
\newblock \bibinfo{journal}{Accident Analysis \& Prevention} \bibinfo{volume}{158}, \bibinfo{pages}{106217}.
\bibitem[{Ghiasi et~al.(2017)Ghiasi, Hussain, Qian and Li}]{ghiasi2017mixed}
\bibinfo{author}{Ghiasi, A.}, \bibinfo{author}{Hussain, O.}, \bibinfo{author}{Qian, Z.S.}, \bibinfo{author}{Li, X.}, \bibinfo{year}{2017}.
\newblock \bibinfo{title}{A mixed traffic capacity analysis and lane management model for connected automated vehicles: A markov chain method}.
\newblock \bibinfo{journal}{Transportation Research Part B: Methodological} \bibinfo{volume}{106}, \bibinfo{pages}{266--292}.
\bibitem[{Grant and Boyd(2008)}]{gb08}
\bibinfo{author}{Grant, M.}, \bibinfo{author}{Boyd, S.}, \bibinfo{year}{2008}.
\newblock \bibinfo{title}{Graph implementations for nonsmooth convex programs}, in: \bibinfo{editor}{Blondel, V.}, \bibinfo{editor}{Boyd, S.}, \bibinfo{editor}{Kimura, H.} (Eds.), \bibinfo{booktitle}{Recent Advances in Learning and Control}. \bibinfo{publisher}{Springer-Verlag Limited}. Lecture Notes in Control and Information Sciences, pp. \bibinfo{pages}{95--110}.
\newblock \bibinfo{note}{\url{http://stanford.edu/~boyd/graph_dcp.html}}.
\bibitem[{ISO(2019)}]{iso2019pas}
\bibinfo{author}{ISO}, \bibinfo{year}{2019}.
\newblock \bibinfo{title}{Iso/pas 21448-road vehicles-safety of the intended functionality}.
\newblock \bibinfo{journal}{International Organization for Standardization} .
\bibitem[{Knoop et~al.(2008)Knoop, Hoogendoorn and Van~Zuylen}]{knoop2008capacity}
\bibinfo{author}{Knoop, V.L.}, \bibinfo{author}{Hoogendoorn, S.P.}, \bibinfo{author}{Van~Zuylen, H.J.}, \bibinfo{year}{2008}.
\newblock \bibinfo{title}{Capacity reduction at incidents: Empirical data collected from a helicopter}.
\newblock \bibinfo{journal}{Transportation research record} \bibinfo{volume}{2071}, \bibinfo{pages}{19--25}.
\bibitem[{Laval and Leclercq(2010)}]{laval2010mechanism}
\bibinfo{author}{Laval, J.A.}, \bibinfo{author}{Leclercq, L.}, \bibinfo{year}{2010}.
\newblock \bibinfo{title}{A mechanism to describe the formation and propagation of stop-and-go waves in congested freeway traffic}.
\newblock \bibinfo{journal}{Philosophical Transactions of the Royal Society A: Mathematical, Physical and Engineering Sciences} \bibinfo{volume}{368}, \bibinfo{pages}{4519--4541}.
\bibitem[{Laval et~al.(2014)Laval, Toth and Zhou}]{laval2014parsimonious}
\bibinfo{author}{Laval, J.A.}, \bibinfo{author}{Toth, C.S.}, \bibinfo{author}{Zhou, Y.}, \bibinfo{year}{2014}.
\newblock \bibinfo{title}{A parsimonious model for the formation of oscillations in car-following models}.
\newblock \bibinfo{journal}{Transportation Research Part B: Methodological} \bibinfo{volume}{70}, \bibinfo{pages}{228--238}.
\bibitem[{Li and Ouyang(2011)}]{li2011characterization}
\bibinfo{author}{Li, X.}, \bibinfo{author}{Ouyang, Y.}, \bibinfo{year}{2011}.
\newblock \bibinfo{title}{Characterization of traffic oscillation propagation under nonlinear car-following laws}.
\newblock \bibinfo{journal}{Procedia-Social and Behavioral Sciences} \bibinfo{volume}{17}, \bibinfo{pages}{663--682}.
\bibitem[{Li et~al.(2024)Li, Zhou, Chen and Zhang}]{li2024disturbances}
\bibinfo{author}{Li, Z.}, \bibinfo{author}{Zhou, Y.}, \bibinfo{author}{Chen, D.}, \bibinfo{author}{Zhang, Y.}, \bibinfo{year}{2024}.
\newblock \bibinfo{title}{Disturbances and safety analysis of linear adaptive cruise control for cut-in scenarios: A theoretical framework}.
\newblock \bibinfo{journal}{Transportation Research Part C: Emerging Technologies} \bibinfo{volume}{168}, \bibinfo{pages}{104576}.
\bibitem[{L{\"{o}}fberg(2004)}]{Lofberg2004}
\bibinfo{author}{L{\"{o}}fberg, J.}, \bibinfo{year}{2004}.
\newblock \bibinfo{title}{Yalmip : A toolbox for modeling and optimization in matlab}, in: \bibinfo{booktitle}{In Proceedings of the CACSD Conference}, \bibinfo{address}{Taipei, Taiwan}.
\bibitem[{Ma and Zhang(2021)}]{ma2021drivers}
\bibinfo{author}{Ma, Z.}, \bibinfo{author}{Zhang, Y.}, \bibinfo{year}{2021}.
\newblock \bibinfo{title}{Drivers trust, acceptance, and takeover behaviors in fully automated vehicles: Effects of automated driving styles and driver’s driving styles}.
\newblock \bibinfo{journal}{Accident Analysis \& Prevention} \bibinfo{volume}{159}, \bibinfo{pages}{106238}.
\bibitem[{Mahmassani(2016)}]{mahmassani201650th}
\bibinfo{author}{Mahmassani, H.S.}, \bibinfo{year}{2016}.
\newblock \bibinfo{title}{50th anniversary invited article—autonomous vehicles and connected vehicle systems: Flow and operations considerations}.
\newblock \bibinfo{journal}{Transportation Science} \bibinfo{volume}{50}, \bibinfo{pages}{1140--1162}.
\bibitem[{Mauch and J.~Cassidy(2002)}]{mauch2002freeway}
\bibinfo{author}{Mauch, M.}, \bibinfo{author}{J.~Cassidy, M.}, \bibinfo{year}{2002}.
\newblock \bibinfo{title}{Freeway traffic oscillations: observations and predictions}, in: \bibinfo{booktitle}{Transportation and Traffic Theory in the 21st Century: Proceedings of the 15th International Symposium on Transportation and Traffic Theory, Adelaide, Australia, 16-18 July 2002}, \bibinfo{organization}{Emerald Group Publishing Limited}. pp. \bibinfo{pages}{653--673}.
\bibitem[{McDonald et~al.(2019)McDonald, Alambeigi, Engstr{\"o}m, Markkula, Vogelpohl, Dunne and Yuma}]{mcdonald2019toward}
\bibinfo{author}{McDonald, A.D.}, \bibinfo{author}{Alambeigi, H.}, \bibinfo{author}{Engstr{\"o}m, J.}, \bibinfo{author}{Markkula, G.}, \bibinfo{author}{Vogelpohl, T.}, \bibinfo{author}{Dunne, J.}, \bibinfo{author}{Yuma, N.}, \bibinfo{year}{2019}.
\newblock \bibinfo{title}{Toward computational simulations of behavior during automated driving takeovers: a review of the empirical and modeling literatures}.
\newblock \bibinfo{journal}{Human factors} \bibinfo{volume}{61}, \bibinfo{pages}{642--688}.
\bibitem[{Milan{\'e}s et~al.(2013)Milan{\'e}s, Shladover, Spring, Nowakowski, Kawazoe and Nakamura}]{milanes2013cooperative}
\bibinfo{author}{Milan{\'e}s, V.}, \bibinfo{author}{Shladover, S.E.}, \bibinfo{author}{Spring, J.}, \bibinfo{author}{Nowakowski, C.}, \bibinfo{author}{Kawazoe, H.}, \bibinfo{author}{Nakamura, M.}, \bibinfo{year}{2013}.
\newblock \bibinfo{title}{Cooperative adaptive cruise control in real traffic situations}.
\newblock \bibinfo{journal}{IEEE Transactions on intelligent transportation systems} \bibinfo{volume}{15}, \bibinfo{pages}{296--305}.
\bibitem[{Mohajerpoor and Ramezani(2019)}]{mohajerpoor2019mixed}
\bibinfo{author}{Mohajerpoor, R.}, \bibinfo{author}{Ramezani, M.}, \bibinfo{year}{2019}.
\newblock \bibinfo{title}{Mixed flow of autonomous and human-driven vehicles: Analytical headway modeling and optimal lane management}.
\newblock \bibinfo{journal}{Transportation research part C: emerging technologies} \bibinfo{volume}{109}, \bibinfo{pages}{194--210}.
\bibitem[{Norris(1998)}]{norris1998markov}
\bibinfo{author}{Norris, J.R.}, \bibinfo{year}{1998}.
\newblock \bibinfo{title}{Markov chains}.
\newblock \bibinfo{number}{2}, \bibinfo{publisher}{Cambridge university press}.
\bibitem[{Oh and Yeo(2015)}]{oh2015impact}
\bibinfo{author}{Oh, S.}, \bibinfo{author}{Yeo, H.}, \bibinfo{year}{2015}.
\newblock \bibinfo{title}{Impact of stop-and-go waves and lane changes on discharge rate in recovery flow}.
\newblock \bibinfo{journal}{Transportation Research Part B: Methodological} \bibinfo{volume}{77}, \bibinfo{pages}{88--102}.
\bibitem[{ORAD(2021)}]{on2021taxonomy}
\bibinfo{author}{ORAD}, \bibinfo{year}{2021}.
\newblock \bibinfo{title}{Taxonomy and definitions for terms related to driving automation systems for on-road motor vehicles}.
\newblock \bibinfo{publisher}{SAE international}.
\bibitem[{Qin et~al.(2025)Qin, Luo and Wang}]{qin2025markov}
\bibinfo{author}{Qin, Y.}, \bibinfo{author}{Luo, Q.}, \bibinfo{author}{Wang, H.}, \bibinfo{year}{2025}.
\newblock \bibinfo{title}{Markov chain-based capacity modeling for mixed traffic flow with bi-class connected vehicle platoons on minor road at priority intersections}.
\newblock \bibinfo{journal}{Physica A: Statistical Mechanics and its Applications} \bibinfo{volume}{658}, \bibinfo{pages}{130301}.
\bibitem[{Shao et~al.(2017)Shao, Wang, Liu, Tang, Li, Zhang and Shen}]{shao2017sigmoid}
\bibinfo{author}{Shao, X.}, \bibinfo{author}{Wang, H.}, \bibinfo{author}{Liu, J.}, \bibinfo{author}{Tang, J.}, \bibinfo{author}{Li, J.}, \bibinfo{author}{Zhang, X.}, \bibinfo{author}{Shen, C.}, \bibinfo{year}{2017}.
\newblock \bibinfo{title}{Sigmoid function based integral-derivative observer and application to autopilot design}.
\newblock \bibinfo{journal}{Mechanical Systems and Signal Processing} \bibinfo{volume}{84}, \bibinfo{pages}{113--127}.
\bibitem[{Shladover et~al.(2012)Shladover, Su and Lu}]{shladover2012impacts}
\bibinfo{author}{Shladover, S.E.}, \bibinfo{author}{Su, D.}, \bibinfo{author}{Lu, X.Y.}, \bibinfo{year}{2012}.
\newblock \bibinfo{title}{Impacts of cooperative adaptive cruise control on freeway traffic flow}.
\newblock \bibinfo{journal}{Transportation Research Record} \bibinfo{volume}{2324}, \bibinfo{pages}{63--70}.
\bibitem[{Sontag(2013)}]{sontag2013mathematical}
\bibinfo{author}{Sontag, E.D.}, \bibinfo{year}{2013}.
\newblock \bibinfo{title}{Mathematical control theory: deterministic finite dimensional systems}. volume~\bibinfo{volume}{6}.
\newblock \bibinfo{publisher}{Springer Science \& Business Media}.
\bibitem[{Sugiyama et~al.(2008)Sugiyama, Fukui, Kikuchi, Hasebe, Nakayama, Nishinari, Tadaki and Yukawa}]{sugiyama2008traffic}
\bibinfo{author}{Sugiyama, Y.}, \bibinfo{author}{Fukui, M.}, \bibinfo{author}{Kikuchi, M.}, \bibinfo{author}{Hasebe, K.}, \bibinfo{author}{Nakayama, A.}, \bibinfo{author}{Nishinari, K.}, \bibinfo{author}{Tadaki, S.i.}, \bibinfo{author}{Yukawa, S.}, \bibinfo{year}{2008}.
\newblock \bibinfo{title}{Traffic jams without bottlenecks—experimental evidence for the physical mechanism of the formation of a jam}.
\newblock \bibinfo{journal}{New journal of physics} \bibinfo{volume}{10}, \bibinfo{pages}{033001}.
\bibitem[{Swaroop and Hedrick(1996)}]{swaroop1996string}
\bibinfo{author}{Swaroop, D.}, \bibinfo{author}{Hedrick, J.K.}, \bibinfo{year}{1996}.
\newblock \bibinfo{title}{String stability of interconnected systems}.
\newblock \bibinfo{journal}{IEEE transactions on automatic control} \bibinfo{volume}{41}, \bibinfo{pages}{349--357}.
\bibitem[{Talebpour and Mahmassani(2016)}]{talebpour2016influence}
\bibinfo{author}{Talebpour, A.}, \bibinfo{author}{Mahmassani, H.S.}, \bibinfo{year}{2016}.
\newblock \bibinfo{title}{Influence of connected and autonomous vehicles on traffic flow stability and throughput}.
\newblock \bibinfo{journal}{Transportation research part C: emerging technologies} \bibinfo{volume}{71}, \bibinfo{pages}{143--163}.
\bibitem[{Ward and Wilson(2011)}]{ward2011criteria}
\bibinfo{author}{Ward, J.A.}, \bibinfo{author}{Wilson, R.E.}, \bibinfo{year}{2011}.
\newblock \bibinfo{title}{Criteria for convective versus absolute string instability in car-following models}.
\newblock \bibinfo{journal}{Proceedings of the Royal Society A: Mathematical, Physical and Engineering Sciences} \bibinfo{volume}{467}, \bibinfo{pages}{2185--2208}.
\bibitem[{Younes and Simmons(2004)}]{younes2004solving}
\bibinfo{author}{Younes, H.L.}, \bibinfo{author}{Simmons, R.G.}, \bibinfo{year}{2004}.
\newblock \bibinfo{title}{Solving generalized semi-markov decision processes using continuous phase-type distributions}, in: \bibinfo{booktitle}{AAAI}, p. \bibinfo{pages}{742}.
\bibitem[{Yue et~al.(2023)Yue, Li, Zhou and Zhang}]{yue2023markov}
\bibinfo{author}{Yue, X.}, \bibinfo{author}{Li, Z.}, \bibinfo{author}{Zhou, Y.}, \bibinfo{author}{Zhang, Y.}, \bibinfo{year}{2023}.
\newblock \bibinfo{title}{Markov-based analytical approximation for mixed traffic delay of signalized intersections}, in: \bibinfo{booktitle}{2023 IEEE 26th International Conference on Intelligent Transportation Systems (ITSC)}, \bibinfo{organization}{IEEE}. pp. \bibinfo{pages}{5140--5145}.
\bibitem[{Zhang et~al.(2019)Zhang, De~Winter, Varotto, Happee and Martens}]{zhang2019determinants}
\bibinfo{author}{Zhang, B.}, \bibinfo{author}{De~Winter, J.}, \bibinfo{author}{Varotto, S.}, \bibinfo{author}{Happee, R.}, \bibinfo{author}{Martens, M.}, \bibinfo{year}{2019}.
\newblock \bibinfo{title}{Determinants of take-over time from automated driving: A meta-analysis of 129 studies}.
\newblock \bibinfo{journal}{Transportation research part F: traffic psychology and behaviour} \bibinfo{volume}{64}, \bibinfo{pages}{285--307}.
\bibitem[{Zhong et~al.(2025)Zhong, Zhou, Kamaraj, Zhou, Kontar, Negrut, Lee and Ahn}]{zhong2025human}
\bibinfo{author}{Zhong, X.}, \bibinfo{author}{Zhou, Y.}, \bibinfo{author}{Kamaraj, A.V.}, \bibinfo{author}{Zhou, Z.}, \bibinfo{author}{Kontar, W.}, \bibinfo{author}{Negrut, D.}, \bibinfo{author}{Lee, J.D.}, \bibinfo{author}{Ahn, S.}, \bibinfo{year}{2025}.
\newblock \bibinfo{title}{Human-automated vehicle interactions: Voluntary driver intervention in car-following}.
\newblock \bibinfo{journal}{Transportation Research Part C: Emerging Technologies} \bibinfo{volume}{171}, \bibinfo{pages}{104969}.
\bibitem[{Zhou et~al.(1996)Zhou, Doyle and Glover}]{zhou1996robust}
\bibinfo{author}{Zhou, K.}, \bibinfo{author}{Doyle, J.}, \bibinfo{author}{Glover, K.}, \bibinfo{year}{1996}.
\newblock \bibinfo{title}{Robust and optimal control} .
\bibitem[{Zhou et~al.(2020)Zhou, Ahn, Wang and Hoogendoorn}]{zhou2020stabilizing}
\bibinfo{author}{Zhou, Y.}, \bibinfo{author}{Ahn, S.}, \bibinfo{author}{Wang, M.}, \bibinfo{author}{Hoogendoorn, S.}, \bibinfo{year}{2020}.
\newblock \bibinfo{title}{Stabilizing mixed vehicular platoons with connected automated vehicles: An h-infinity approach}.
\newblock \bibinfo{journal}{Transportation Research Part B: Methodological} \bibinfo{volume}{132}, \bibinfo{pages}{152--170}.

\end{thebibliography}





\end{document}